\documentclass[11pt]{article}

\usepackage{amsmath}
\usepackage{amssymb}
\usepackage{amsthm}
\usepackage{graphicx}
\usepackage{pstricks}

\makeatletter
 \newtheorem{theorem}{Theorem}[section]
 \newtheorem{lemma}[theorem]{Lemma}
 \newtheorem{proposition}[theorem]{Proposition}
 \newtheorem{corollary}[theorem]{Corollary}
 \newtheorem*{theorem*}{Theorem}
\theoremstyle{definition}
 \newtheorem{definition}[theorem]{Definition}
\theoremstyle{remark}
 \newtheorem{remark}{Remark} 
 \newtheorem{example}{Example} 
 
\newcommand{\la}{\langle}
\newcommand{\ra}{\rangle}
\newcommand{\HdC}{\mathbf H^{2}_{\mathbb C}}

\newcommand{\bp}{{\bf p}}

\newcommand{\bn}{{\bf n}}
\newcommand{\bq}{{\bf q}}
\newcommand{\br}{{\bf r}}

\renewcommand{\arg}{\mbox{arg}}

\newcommand{\A}{{{\mathbb A}}}
\newcommand{\X}{{{\mathbb X}}}
\newcommand{\C}{{{\mathbb C}}}

\newcommand{\R}{{{\mathbb R}}}
\newcommand{\Z}{{{\mathbb Z}}}

\newcommand{\tr}{{{\rm tr}}}

\begin{document}

\pagenumbering{arabic}
\title{Complex hyperbolic free groups with many parabolic elements}
\author{John R Parker \\
Department of Mathematical Sciences\\
Durham University\\
South Road\\
Durham 3H1 3LE, England\\
e-mail: {\tt j.r.parker@durham.ac.uk}
\and 
Pierre Will\footnote{Author supported by ANR project SGT. This work was 
done during stays of the first author in Grenoble, and of the second in 
Durham that were funded by ANR project SGT.}\\
Institut Fourier\\
Universit\'e Grenoble I\\
100 rue des Maths\\
38042 St-Martin d'H\`eres, France\\
e-mail: {\tt pierre.will@ujf-grenoble.fr}}
%\author{John R Parker and Pierre Will}
\date{\today}
\maketitle

\begin{abstract}
\noindent
We consider in this work representations of the of the fundamental group of 
the 3-punctured sphere in
${\rm PU}(2,1)$ such that the boundary loops are mapped to ${\rm PU}(2,1)$. 
We provide a system of coordinates on the 
corresponding representation variety, and analyse more specifically those 
representations corresponding
to subgroups of $(3,3,\infty)$-groups. In particular we prove that it is 
possible to construct 
representations of the free group of rank two $\la a,b\ra$ in ${\rm PU}(2,1)$ 
for which $a$, $b$, $ab$, $ab^{-1}$, 
$ab^2$, $a^2b$ and $[a,b]$ all are mapped to parabolics. 
\end{abstract}
\maketitle

\section{Introduction}
In this paper we consider representations of $F_2$, the free group of rank two, into ${\rm SU}(2,1)$. The 
latter group is a three-fold covering of ${\rm PU}(2,1)$, which is the holomorphic isometry
group of complex hyperbolic two-space ${\bf H}^2_{\mathbb C}$. The 
deformation space which is the space of conjugacy classes of
representations 
$$
{\mathcal R}={\rm Hom}\bigl(F_2,{\rm SU}(2,1)\bigr)/{\rm SU}(2,1).
$$
It is not hard to see that the dimension of this space is the same
as that of ${\rm SU}(2,1)$, namely four complex dimensions or eight
real dimensions. We will be particularly interested in those representations
with many parabolic elements. The locus of points in ${\mathcal R}$ 
where a given group element is parabolic is an algebraic real hypersurface.
In particular, if $a$ and $b$ is a given generating set for
$F_2=\langle a,b\rangle$ then we shall only consider those representations 
$\rho\in{\mathcal R}$ for which $A=\rho(a)$, $B=\rho(b)$ and 
$AB=\rho(ab)$ are all parabolic. We say that such a representation
of $F_2$ to ${\rm SU}(2,1)$ is \textit{parabolic}. Since these three conditions are independent,
this means that the space of parabolic representations is a five dimensional subspace of ${\mathcal R}$. 
We will very often use the alternate presentation $F_2=\la a,b,c\vert abc=1\ra$, which gives an 
identification of $F_2$ with the fundamental group of the 3-punctured sphere. With this identification, 
parabolic representations of $F_2$ are thus representations of the 3-punctured sphere mapping peripheral
loops to parabolics. It is a well known fact that there is only one such representation in PSL(2,$\C$) up 
to conjugacy. We will describe here the corresponding deformation space for SU(2,1).

Before giving our main results, we now indicate our motivation. There
is a beautiful description of the ${\rm SU}(2,1)$ representation space
of closed surface groups due to Goldman \cite{Gol1,Gol2}, Toledo 
\cite{T} and Xia \cite{X}. 
Of particular interest are {\sl complex hyperbolic quasi-Fuchsian} 
representations of a surface group to ${\rm SU}(2,1)$; see Parker-Platis
\cite{PP} for a survey on this topic. In particular, Parker and Platis,
Problem 6.2 of Parker-Platis \cite{PP}, ask whether the boundary of complex 
hyperbolic quasi-Fuchsian space comprises representations with 
parabolic elements and they ask which parabolic maps can arise. 
We can consider a decomposition of the surface into three-holed spheres
and then allow the three boundary curves to be pinched, so they are
represented by parabolic elements. The fundamental group of a
three holed sphere is a free group on two generators 
$F_2=\langle a,b\rangle$. The condition that the three boundary curves are
pinched is exactly that $A=\rho(a)$, $B=\rho(b)$ and 
$AB=\rho(ab)$ should all be parabolic. 

If $AB$ is parabolic then, of course, the product $BA$ is 
parabolic as well, and considering its fixed point we obtain an ideal 
tetrahedron in $\HdC$ (an ordered $4$-tuple of boundary points). 
The shape of the tetrahedron $\tau_\rho=(p_A,p_B,p_{AB},p_{BA})$ is a conjugacy 
invariant of the representation $\rho$ that we are going to use to give a 
coordinate system on the family of conjugacy classes of representations. 
Moreover the shape of a tetrahedron $\tau_\rho$ for a parabolic 
representation $\rho$ can not be arbitrary. Indeed we prove that if $\rho$ 
is a parabolic representation of $F_2$ to ${\rm SU}(2,1)$ then $\tau_\rho$ 
is \textit{balanced}. An ideal tetrahedron $(p_1,p_2,p_3,p_4)$ is balanced when 
$p_3$ and $p_4$ are mapped to the same point by the orthogonal projection 
onto the geodesic connecting $p_1$ and $p_2$.  To see this, we connect the 
shape of the tetrahedron to the conjugacy classes of $\rho(a)$, $\rho(b)$ and 
$\rho(ab)$ via the complex cross-ratio  $\X(p_A,p_B,p_{AB},p_{BA})$ 
(see \cite{KR}). 
More precisely, we prove in Corollary \ref{3spherecrossratio} that when 
$\rho$ is parabolic we have:
\begin{equation}\label{eq-cr-eval}
 \X(p_A,p_B,p_{AB},p_{BA})=\lambda_A\lambda_B\lambda_C,
\end{equation}
where $C=(AB)^{-1}$ and $\lambda_A$, $\lambda_B$ and $\lambda_C$ are 
respectively the eigenvalues associated to the  boundary fixed points of 
$A$, $B$ and $C$. As $A$, $B$ and $C$ are parabolic, these eigenvalues all 
have unit modulus, which implies that the cross-ratio also has unit modulus. 
This condition is equivalent to saying that the tetrahedron $\tau_\rho$ is 
balanced, as proved in Section \ref{balancedsection}.

The next question is the converse. Given a balanced ideal tetrahedron $\tau$, 
and given three 
unit complex numbers $\lambda_A$, $\lambda_B$ and $\lambda_C$ such that 
\eqref{eq-cr-eval} holds, can 
we construct a parabolic representation $\rho:F_2\longrightarrow{\rm PU}(2,1)$ 
such that $\tau=\tau_\rho$ as before?  The answer is yes, if we allow that $A$, 
$B$ and $C$ are parabolic or complex reflections. This ambiguity comes from 
the fact that an isometry having a boundary fixed point with unit modulus 
eigenvalue can be either parabolic or a complex reflection (see section 
\ref{section-conj-class}). This is Proposition \ref{repEP}.

We focus next on the case where the three (unit modulus) eigenvalues 
$\lambda_A$, $\lambda_B$ and $\lambda_C$ all are equal. From \eqref{eq-cr-eval}
they are necessarily all the same cube root of the cross ratio. 
We show that such a representation admits a three fold symmetry.
In particular, it is a subgroup of a $(3,3,\infty)$ group generated
by two regular elliptic maps $J_1$ and $J_2$ or order 3 whose product
$J_1J_2$ is parabolic.
Specifically, we prove (Theorem \ref{thm-symmetric-groups}):

\begin{theorem*}
Suppose that $\rho: F_2=\langle a,b,c\ |\ abc=id\rangle\longrightarrow
{\rm SU}(2,1)$ has the property that $\rho(a)$, $\rho(b)$, $\rho(c)$ are
all parabolic and have the same eigenvalues. Then $\rho(F_2)$ is an
index 3 subgroup of a ${\rm SU}(2,1)$ representation of the $(3,3,\infty)$
group.
\end{theorem*}

This leads to our main result connecting the representation to
geometry of complex hyperbolic space, (Theorem \ref{thm-main1}):

\begin{theorem*}
There is a bijection between the set of ${\rm PU}(2,1)$-orbits of 
non-degenerate balanced ideal tetrahedra, and the set of 
${\rm PU}(2,1)$-conjugacy classes of $(3,3,\infty)$ groups in ${\rm PU}(2,1)$.
\end{theorem*}
Using a normalisation of balanced tetrahedra, we obtain an explicit 
parametrisation of the order 3 generators of $\la 3,3,\infty\ra$-group. 
Next, we investigate when more group elements are parabolic. In particular, we can prove
(Corollary \ref{cor-superpinching}):
\begin{theorem*}
There is a one parameter family of groups generated $A$ and $B$
in ${\rm SU}(2,1)$ so that $A$, $B$, $AB$, $AB^{-1}$, $AB^2$, 
$A^2B$ and $[A,B]$ are all parabolic.
\end{theorem*}

It would be very interesting to find out whether any (or all) of these 
representations are discrete and free, and also whether or not
it is possible to find any more parabolic elements.

\section{Fixed point tetrahedra of thrice punctured sphere groups}

We refer the reader to \cite{ChGr,Goldbook,Parbook} for basic material 
on the complex hyperbolic space. We will denote by $\A$ and $\X$ respectively 
the Cartan 
invariant (see Chapter 7 of \cite{Goldbook}) and the complex cross-ratio 
(see \cite{KR}, and Chapter 7 of \cite{Goldbook}).

\subsection{Conjugacy classes in ${\rm PU}(2,1)$.
\label{section-conj-class}}
We recall that the group of holomorphic isometries of the complex hyperbolic 
is ${\rm PU}(2,1)$. Elements of ${\rm PU}(2,1)$ are classified by the usual 
trichotomy : loxodromic, elliptic and parabolic isometries. As in the 
classical cases of ${\rm PSL}(2,\R)$ and ${\rm PSL}(2,\C)$ it is 
possible to detect the types using the trace of a lift of an element of 
${\rm PU}(2,1)$ to ${\rm SU}(2,1)$. However certain subtleties arise here 
that we would like to describe as they will play a role in our work. Let us 
first recall the trace classification of isometries (Theorem 6.4.2 of 
\cite{Goldbook}).

\begin{proposition}\label{prop-class-trace}
Let $A$ be a non trivial element of ${\rm PU}(2,1)$ and ${\bf A}$ a lift of 
it to ${\rm SU}(2,1)$. We denote by $f$ the polynomial function given by 
$f(z)=|z|^4 - 8{\rm Re}(z^3)+18|z |^2-27$. Then
\begin{enumerate}
\item The isometry $A$ is loxodromic if and only if $f(\tr {\bf A})>0$.
\item It is \textit{regular elliptic} if and only if $f(\tr {\bf A})<0$.
\item It is parabolic or a complex reflection if and only if $f(\tr {\bf A})=0$.
\end{enumerate}
\end{proposition}

An elliptic isometry $A$ is called \textit{regular} if and only if any 
lift ${\bf A}$ to ${\rm SU}(2,1)$ has three pairwise distinct eigenvalues. 
Whenever an elliptic isometry is not regular it is called a 
\textit{complex reflection}. The set of fixed points in $\HdC$ of a complex 
reflection can be either 
a point or complex line (see \cite{Goldbook} for details). Note that a complex 
reflection does not necessarily 
have finite order, contrary to the usual terminology in real spaces.

A parabolic isometry $P$ is called unipotent whenever it admits a unipotent 
lift ${\bf P}\in{\rm SU}(2,1)$. There are two conjugacy classes of unipotent 
parabolics, namely 2-step or 3-step unipotents, depending on the nilpotency 
index of ${\bf P} - I$. A non unipotent parabolic map is called 
\textit{screw-parabolic}. The spectrum of the lift of a parabolic is always 
of the kind $\lbrace e^{i\alpha},e^{i\alpha},e^{-2i\alpha}\rbrace$ for some 
$\alpha\in\R$. When $\alpha=0$, the parabolic is unipotent. Therefore the 
traces of parabolic isometries form a curve in $\C$, given by 
$\lbrace 2e^{i\alpha}+e^{-2i\alpha},\alpha\in\R\rbrace$, which is depicted in 
Figure \ref{parabcurve}. We will often refer to this curve as \textit{the 
deltoid}.
In view of Proposition \ref{prop-class-trace}, 
this curve is the zero-locus of the polynomial $f$. However 
Proposition \ref{prop-class-trace} tells us that if $f(\tr {\bf A})=0$, 
then we need more information to know the type of the isometry $A$, as it 
could be a complex reflection. This can be done by using the fact that lifts 
of complex reflections  are semi-simple whereas those of parabolics are not 
(see the proof of Proposition \ref{u-isom}). 

\begin{figure}
\begin{center}
\includegraphics[scale=0.3]{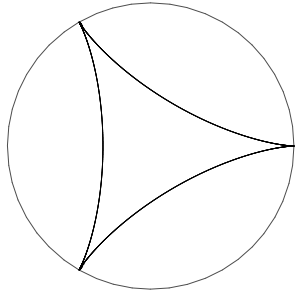}
\caption{The null locus of $f$ and the circle $\{|z|=3\}$.\label{parabcurve}} 
\end{center}
\end{figure}

\subsection{Fixed points, eigenvalues, cross-ratios}
\begin{definition}
  We will call \textit{parabolic} any representation
  $\rho:F_2=\la a,b,c\ :\ abc=id\ra\longrightarrow {\rm PU}(2,1)$ 
which maps $a$, $b$ and $c$ (thus
  $ab$ and $ba$) to parabolic isometries. We will denote by
  $\mathcal{P}$ the set of parabolic representations of $F_2$.
\end{definition}
We will denote by $A$, $B$ and $C$ the images by $\rho$ of $a$, $b$ and $c$, 
and by $p_A$, $p_B$ and $p_C$ 
their boundary fixed points (for a parabolic representation).
\begin{definition}
  Let $\rho:F_2\longrightarrow {\rm PU}(2,1)$ be a parabolic representation. We will call 
\textit{fixed point tetrahedron of $\rho$} and denote by $\tau_{\rho}$ the ideal tetrahedron
  $(p_A,p_B,p_{AB},p_{BA})$.
\end{definition}
If $A$ is in ${\rm PU}(2,1)$, with fixed point $p$, we will denote by the same
letter in bold font $\bp$ a lift of $p$ to $\C^3$.
\begin{definition}
  If $A\in{\rm SU}(2,1)$ fixes projectively $p_A$, we say that $\lambda_A$ is
  the eigenvalue of $A$ associated to $p$ if $A\bp_A=\lambda_A\bp_A$.
\end{definition}
The following lemma provides an identity connecting eigenvalues with 
cross ratios and angular invariant of fixed points that will play an 
important role in our discussion. We refer the reader to \cite{KR} or to 
Chapter 7 of \cite{Goldbook} for the basic definitions concerning the 
Kor\'anyi-Riemann cross-ratio of four points, which we will denote by $\X$
and the Cartan angular invariant of three points, which we denote by $\A$ . 
\begin{lemma}\label{eigencross}
  Let $A$ and $B$ be in ${\rm PU}(2,1)$. Let $p_A$ and $p_B$ be fixed points
  of $A$ and $B$ with eigenvalues $\lambda_A$ and $\lambda_B$. Let
  $p_{AB}$ and $p_{BA}$ be fixed points of $AB$ and $BA$ such that
  $Ap_{BA}=p_{AB}$. Denote by $\lambda_{AB}$ the corresponding
  eigenvalue of $AB$. Assume that the four points
  $(p_A,p_B,p_{AB},p_{BA})$ are pairwise distinct.
\begin{enumerate}
\item The eigenvalues of $AB$ and $BA$ associated with $p_{AB}$ and
  $p_{BA}$ are equal.
\item The four points $p_A$, $p_B$, $p_{AB}$, $p_{BA}$ satisfy the
  following cross-ratio identity.
\begin{equation}
\X(p_A,p_B,p_{AB},p_{BA})
=\dfrac{1}{\overline{\lambda}_A\overline{\lambda}_B\lambda_{AB}}
\end{equation}
\item Taking the principal determination of the argument, we have 
$$\arg\left(\X(p_A,p_B,p_{AB},p_{BA})\right)=\A(p_A,p_B,p_{AB})-\A(p_A,p_B,p_{BA})\quad [2\pi].$$
\end{enumerate}
\end{lemma}
The last part of Lemma \ref{eigencross} has nothing to do with $A$ and $B$, and is a general property of ideal 
tetrahedra in $\HdC$. One should be careful to write this equality only up to  a multiple of $2\pi$, as noted by 
Cunha and Gusevkii in \cite{CuGu}. 
\begin{proof}
  The first part of the Lemma is a direct consequence of
  $AB=A(BA)A^{-1}$.  Because $Ap_{BA}=p_{AB}$ and $Bp_{AB}=p_{BA}$,
  there exists complex numbers $\mu$ and $\nu$ such that
$$A\bp_{BA}=\mu \bp_{AB}\mbox{ and }B\bp_{AB}=\nu \bp_{BA}.$$ 
But any lift $\bp_{AB}$ of $p_{AB}$ satisfies
$AB\bp_{AB}=\lambda_{AB}\bp_{AB}$.  This implies that
$\lambda_{AB}=\mu\nu$. Let us compute the cross ratio. We use the fact that
$A$ and $B$ preserve the Hermitian form.
\begin{eqnarray*}
\X(p_A,p_B,p_{AB},p_{BA}) 
& = & \dfrac{\la \bp_{AB},\bp_A\ra\la\bp_{BA},\bp_B\ra}
{\la\bp_{AB},\bp_B\ra\la\bp_{BA},\bp_{A}\ra} \\
& = & \dfrac{\la \bp_{AB},\bp_A\ra\la\bp_{BA},\bp_B\ra}
{\la B\bp_{AB},B\bp_B\ra\la A\bp_{BA},A\bp_{A}\ra} \\
& = &\dfrac{\la \bp_{AB},\bp_A\ra\la\bp_{BA},\bp_B\ra}
{\overline{\lambda}_A\overline{\lambda}_B\mu\nu\la \bp_{BA},
\bp_B\ra\la\bp_{AB},\bp_A\ra}\\
& = & \dfrac{1}{\overline{\lambda}_A\overline{\lambda}_B\lambda_{AB}}.
\end{eqnarray*}
Finally, 
$$
\X(p_A,p_B,p_{AB},p_{BA}) 
=\dfrac{
|\la\bp_{BA},\bp_B\ra|^2
\la \bp_{AB},\bp_A\ra\la\bp_A,\bp_B\ra\la\bp_B,\bp_{AB}\ra}
{|\la\bp_{AB},\bp_B\ra|^2
\la\bp_{BA},\bp_{A}\ra\la\bp_A,\bp_B\ra\la\bp_B,\bp_{BA}\ra}. 
$$
The result follows by taking argument on both sides since, by definition we have:
\begin{eqnarray*}
\A(p_A,p_B,p_{AB}) & = & \arg\bigl(-
\la \bp_{AB},\bp_A\ra\la\bp_A,\bp_B\ra\la\bp_B,\bp_{AB}\ra\bigr), \\
\A(p_A,p_B,p_{BA}) & = & \arg\bigl(-
\la\bp_{BA},\bp_{A}\ra\la\bp_A,\bp_B\ra\la\bp_B,\bp_{BA}\ra\bigr).
\end{eqnarray*}
\end{proof}
Let us rephrase Lemma \ref{eigencross} for a parabolic representation.
\begin{corollary}\label{3spherecrossratio}
  Let $A$ and $B$ be two parabolic isometries such that $AB$ (and thus
  $BA$) are both parabolic with fixed points on $\partial\HdC$
  $p_{A}$, $p_B$, $p_{AB}$ and $p_{BA}$. Then 
\begin{enumerate}
\item The cross ratio $\X(p_A,p_B,p_{AB}, p_{BA})$ has unit
  modulus. 
\item Moreover, setting $C=(AB)^{-1}$ and denoting by $\lambda_C$ the
  eigenvalue of $C$ associated with $p_{AB}$ it holds
  $\X(p_A,p_B,p_{AB},p_{BA})=\lambda_A\lambda_B\lambda_C.$
 \end{enumerate}
\end{corollary}
\begin{proof}
It is a direct consequence of $\lambda_C=\lambda_{AB}^{-1}$ and of the fact 
that eigenvalues of parabolics are unit complex numbers.
\end{proof}
\subsection{Balanced ideal tetrahedra.\label{balancedsection}}
\begin{definition}
  Let $\tau=(p_1,p_2,p_3,p_4)$ be an ideal tetrahedron and $\pi_{12}$
  be the orthogonal projection onto the (real) geodesic
  $\gamma_{12}=(p_1p_2)$. We will say that $\tau$ is \textit{balanced}
  whenever the images of $p_3$ and $p_4$ under $\pi_{12}$ are equal.
\end{definition}
\begin{definition}
We denote by $\mathcal{B}$ the set of balanced ideal tetrahedra.
\end{definition}
We will use the following choice for the cross ratio:
$$\X(p_1,p_2,p_3,p_4)=\dfrac{\la\bp_3,\bp_1\ra\la\bp_4,\bp_2\ra}
{\la\bp_3,\bp_2\ra\la\bp_4,\bp_1\ra}.$$
\begin{proposition}\label{unitCR}
An ideal tetrahedron $\tau$ is balanced if and only if the cross-ratio
$\X(p_1,p_2,p_3,p_4)$ has unit modulus.
\end{proposition}
An immediate corollary of Proposition \ref{unitCR} and Lemma \ref{eigencross}
is:
\begin{corollary}
 Let $\rho: F_2\longrightarrow{\rm PU}(2,1)$ be a parabolic representation. 
The tetrahedron $\tau_\rho$ is balanced.
\end{corollary}

\begin{proof}(Proposition \ref{unitCR}.).
  Choose lifts $\bp_1$ and $\bp_2$ of $p_1$ and $p_2$ in such a way
  that $\la \bp_1,\bp_2\ra=-1$. Then the projection of a point $z$ in the
  closure of $\HdC$ onto $\gamma_{12}$ is given by
\begin{equation}
  \pi_{12}(z)=\sqrt{\dfrac{|\la {\bf z},\bp_2\ra|}
{|\la {\bf z},\bp_1\ra|}}\bp_1+\sqrt{\dfrac{|\la {\bf z},\bp_1\ra|}
{|\la {\bf z},\bp_2\ra|}}\bp_2.
\end{equation}
Note that this expression does not depend on the chosen lift for $z$. Therefore
the condition $\pi_{12}(p_3)=\pi_{12}(p_4)$ is equivalent to the two
relations obtained by identifying the $\bp_1$ and $\bp_2$ components
of $\pi_{12}(p_3)$ and $\pi_{12}(p_4)$. This gives (after squaring
both sides of the equality)

$$ \dfrac{|\la\bp_3,\bp_2\ra|}{|\la\bp_3,\bp_1\ra|} 
= \dfrac{|\la\bp_4,\bp_2\ra|}{|\la\bp_4,\bp_1\ra|}\, \mbox{ and } \,
  \dfrac{|\la\bp_3,\bp_1\ra|}{|\la\bp_3,\bp_2\ra|}
 = \dfrac{|\la\bp_4,\bp_1\ra|}{|\la\bp_4,\bp_2\ra|}
$$

These two relations are clearly both equivalent to

$$ \left|\dfrac{\la\bp_3,\bp_1\ra\la\bp_4,\bp_2\ra}
{\la\bp_3,\bp_2\ra\la\bp_4,\bp_1\ra} \right|=\left|\X(p_1,p_2,p_3,p_4)\right|=1
$$
\end{proof}

By applying an element of ${\rm PU}(2,1)$ if necessary, we may assume that
\begin{equation}\label{standard-points}
\bp_1=\begin{bmatrix}1\\0\\0\end{bmatrix}, \quad 
\bp_2=\begin{bmatrix}0\\0\\1\end{bmatrix}, \quad 
\bp_3=\begin{bmatrix}
-e^{2i\theta} \\ \sqrt{2\cos(2\theta)}e^{i\theta-i\psi} \\ 1 \end{bmatrix}, \quad
\bp_4=\begin{bmatrix}
-r^2e^{-2i\phi} \\ r\sqrt{2\cos(2\phi)}e^{-i\phi+i\psi} \\ 1 \end{bmatrix}
\end{equation}
where $r>0$, $2\theta\in [-\pi/2,\pi/2]$, $2\phi\in[-\pi/2,\pi/2]$ and 
$\psi\in[0,\pi/2]$. The tetrahedron is completely determined up to
${\rm PU}(2,1)$ equivalence by the parameters $r$, $\theta$, $\phi$
and $\psi$. We want to now give an invariant interpretation of these
parameters. First observe that
$2\theta=\A(p_2,p_1,p_3)$ and $2\phi=\A(p_1,p_2,p_4)$. 

\begin{lemma}
In the above normalisation, the tetrahedron $(p_1,p_2,p_3,p_4)$ is
balanced if and only if $r=1$.
\end{lemma}

\begin{proof}
Computing the cross ratio in this case, we obtain 
$\X(p_1,p_2,p_3,p_4)=r^2e^{-2i\theta-2i\phi}.$
\end{proof}

Note that this implies that an ideal tetrahedron $(p_1,p_2,p_3,p_4)$ is 
balanced if and only if the two points $p_3$ and $p_4$ both lie on the 
boundary of a bisector $\mathcal{B}$ whose complex spine is the complex 
line spanned by $p_1$ and $p_2$  and whose real spine is a geodesic 
orthogonal to $(p_1p_2)$ (see Chapter 5 of \cite{Goldbook} for definitions of 
these notions).

\begin{definition}\label{defitau}
We denote by $\tau(\theta,\phi,\psi)$ the tetrahedron given by
(\ref{standard-points}), where $r$ is replaced by $1$ in $p_4$.
\end{definition}

\begin{definition}
Let $(p_1,p_2,p_3,p_4)$ be an ideal tetrahedron
so that neither of the triples $(p_1,p_2,p_3)$ and
$(p_1,p_2,p_4)$ lie in a complex line . Denote by $c_{12}$ a polar
vector to the complex line spanned by $p_1$ and $p_2$. The following quantity 
is well-defined and is called the \textit{bending parameter}.
\begin{equation}\label{eq-bend-invt}
\mathbb{B}(p_1,p_2,p_3,p_4)
=\X(p_4,p_3,p_1,c_{12})\cdot\X(p_4,p_3,p_2,c_{12})
\end{equation}
\end{definition}

To check that $\mathbb{B}$ is well-defined, note fist that it does not 
depend on the choice of lifts for the $p_i$'s, nor on the choice of $c_{12}$. 
Secondly, the Hermitian products involving $c_{12}$ in the two 
cross-ratios are $\la c_{12}, p_3\ra$ and $\la c_{12}, p_4\ra$ which are 
non zero in view of the assumption made, therefore the two cross-ratios 
$ \X(p_4,p_3,p_1,c_{12})$ and $\X(p_4,p_3,p_2,c_{12})$ are well-defined. 

\medskip

\begin{example}
In the  normalised form given above by (\ref{standard-points}), we see that
$$
\mathbb{B}(p_1,p_2,p_3,p_4)=\dfrac{\cos(2\theta)}{\cos(2\phi)}e^{4i\psi}.
$$
\end{example}

\medskip

Assume that $\A(p_2,p_1,p_3)=\pm\pi/2$ or $\A(p_1,p_2,p_4)=\pm\pi/2$. This 
means that 
$p_3$ or $p_4$ respectively lies on the complex line through $p_1$ and $p_2$. 
Using the standard form (\ref{standard-points}) we see that the middle entry 
of $\bp_3$ or $\bp_4$ is zero. Therefore the angle $\psi$ is not well defined 
in that case.

The following proposition is a straightforward consequence of the
above normalisation.

\begin{proposition}\label{prop-tet-params}
  A balanced tetrahedron $(p_1,p_2,p_3,p_4)$ such that neither of the 
  triples $(p_1,p_2,p_3)$ and $(p_1,p_2,p_4)$ lie in a complex line is
  uniquely determined up to ${\rm PU}(2,1)$ by the three quantities 
$\A(p_1,p_2,p_3)$, $\A(p_1,p_2,p_4)$ and $\mathbb{B}(p_1,p_2,p_3,p_4)$.
\end{proposition}

\section{Constructing thrice punctured sphere groups from 
tetrahedra}

We wish now to work in the converse direction: given a balanced ideal
tetrahedron $(p_1,p_2,p_3,p_4)$, is it possible to construct a
parabolic representation $\rho : F_2\longrightarrow {\rm PU}(2,1)$, such
that $\rho(a)$ fixes $p_1$, $\rho(b)$ fixes $p_2$, $\rho(ab)$ fixes
$p_3$ and $\rho(ba)$ fixes $p_4$.

\subsection{Mappings of boundary points}

In this section we will consider configurations of distinct 
points on $\partial\HdC$, and use them to construct maps in 
${\rm Isom}(\HdC)$ with certain properties. Consider a matrix
$A$ in ${\rm SU}(2,1)$. We know that the eigenvectors of $A$
in $V_-$ and $V_0$ correspond to fixed points of $A$ in $\HdC$
and $\partial\HdC$ respectively. We say that $p\in\partial\HdC$
is a {\sl neutral fixed point} of $A$ if the corresponding
eigenvector ${\bf p}$ has an eigenvalue $\lambda$ with $|\lambda|=1$.
Note that a matrix $A$ with a neutral fixed point in $\partial\HdC$
must be either parabolic or a complex reflection.

In particular, we consider triples of points $p$, $q$ and $r$ of 
$\partial\HdC$. Our goal will be to show that there is a
unique holomorphic isometry of $\HdC$ which sends $q$ to $r$
and with $p$ as a neutral fixed point with a prescribed eigenvalue.
Moreover, we will show how to determine when such an isometry is 
parabolic and when it is a complex reflection.

\begin{proposition}\label{u-isom}
Let $p$, $q$, $r$ be distinct points of $\partial\HdC$ and let
$\lambda$ be a complex number of unit modulus. Then there exists
a unique holomorphic isometry $A$ sending $q$ to $r$ and for which
$p$ is a neutral fixed point with associated eigenvalue $\lambda$.
Moreover, 
\begin{enumerate}
\item If $\lambda^3=-e^{2i\A(p,q,r)}$ and $p$, $q$ and $r$
do not lie in a complex line, then $A$ is elliptic.
\item Otherwise $A$ is parabolic.
\end{enumerate}
\end{proposition}

\begin{proof}
  First, such an isometry is unique if it exists. Indeed, if there
  were two such isometries, say $f_1$ and $f_2$, then $f_1\circ
  f_2^{-1}$ would fix both $p$ and $r$. Moreover the eigenvalue of
  $f_1\circ f_2^{-1}$ associated with $p$ would be $1$ (or a cube root
  of $1$). Thus $f_1\circ f_2^{-1}$ would be the identity.

  To prove existence, let us fix lifts $(\bp,\bq,\br)$ for the three
  points $(p,q,r)$. 

\begin{itemize}
\item Assume first that $(\bp,\bq,\br)$ is a basis, that is $(p,q,r)$
  do not lie on a common complex line.  The following matrix, written
  in the basis $(\bp,\bq,\br)$ has eigenvalue $\lambda$ associated to $\bp$
  , preserves the Hermitian product and projectively maps $q$ to $r$.
\begin{equation}\label{casebasis}
M_1=\begin{bmatrix}
\lambda & 0 & 
\lambda\dfrac{\la\br,\bq\ra}{\la\bp,\bq\ra}
+\overline{\lambda}^2\dfrac{\la\br,\bp\ra\la\bq,\br\ra}
{\la\bp,\br\ra\la\bq,\bp\ra}\\
0 & 0 & -\overline{\lambda}^2\dfrac{\la\br,\bp\ra}{\la\bq,\bp\ra}\\
0 & \lambda\dfrac{\la\bq,\bp\ra}{\la\br,\bp\ra} & \lambda+\overline{\lambda}^2
\end{bmatrix}.
\end{equation}
It is not hard to check that $M_1$ preserves the Hermitian form.
Furthermore, this isometry is elliptic if and only if the matrix 
$M_1-\lambda\cdot I$ has rank one.  Now,
\begin{equation}
M_1-\lambda\cdot I=\begin{bmatrix}
0 & 0 & \lambda\dfrac{\la\br,\bq\ra}{\la\bp,\bq\ra}
+\overline{\lambda}^2\dfrac{\la\br,\bp\ra\la\bq,\br\ra}
{\la\bp,\br\ra\la\bq,\bp\ra}\\
0 & -\lambda & -\overline{\lambda}^2\dfrac{\la\br,\bp\ra}{\la\bq,\bp\ra}\\
0 & \lambda\dfrac{\la\bq,\bp\ra}{\la\br,\bp\ra} & \overline{\lambda}^2
.\end{bmatrix}
\end{equation}
Since the bottom right $2\times 2$ minor of $M_1-\lambda\cdot I$ vanishes, we
see that $M_1$ is elliptic if and only if the top right
entry of $M_1-\lambda\cdot I$ vanishes, which gives after a little rewriting
$$
\lambda^3=-\dfrac{\la\bp,\bq\ra\la\bq,\br\ra\la\br,\bp\ra}
{\la\bp,\br\ra\la\br,\bq\ra\la\bq,\br\ra}
=-e^{2i\A(p,q,r)}.
$$
\item If $(\bp,\bq,\br)$ is not a basis of $\C^3$, that is if
  $(p,q,r)$ lie on a complex line $L$, then any isometry fixing $p$
  and mapping $q$ to $r$ preserves $L$. If $\bn$ is polar to $L$, then
  $(\bp,\bn,\bq)$ is a basis of $\C^3$.  In this basis, the vector
  $\br$ is given by
$$
\br=\dfrac{\la\br,\bq\ra}{\la\bp,\bq\ra}\bp
+\dfrac{\la\br,\bp\ra}{\la\bq,\bp\ra}\bq
$$
The matrix $M_2$ given in (\ref{casenotbasis}) represents a
holomorphic isometry mapping $q$ to $r$ and with $p$ a neutral fixed point:
\begin{equation}\label{casenotbasis}
M_2=\begin{bmatrix}
\lambda & 0 & 
\lambda\dfrac{\la\br,\bq\ra\la\bq,\bp\ra}{\la\br,\bp\ra\la\bp,\bq\ra}\\
0 & 1/\lambda^{2} & 0\\
0 & 0 & \lambda
\end{bmatrix}
\end{equation}
Because the three points $p$, $q$, $r$ are distinct, the top-right coefficient
is never zero thus $M_2-\lambda\cdot I$ always has rank 2. Hence
$M_2$ represents a parabolic isometry.
\end{itemize}
\end{proof}

\medskip

As a direct application of Proposition \ref{u-isom}, we can associate
parabolic (or boundary elliptic) representations to balanced ideal tetrahedra.

\begin{proposition}\label{repEP}
Let $(p_1,p_2,p_3,p_4)$ be a balanced ideal tetrahedron and let
$\lambda_A$ and $\lambda_B$ be two complex numbers of modulus 1. 
There exists a unique representation 
$\rho:F_2  \longrightarrow {\rm PU}(2,1)$ such that 
\begin{itemize}
\item $A=\rho(a)$ fixes $p_1$ with eigenvalue $\lambda_A$ and $B=\rho(b)$ fixes
  $p_2$ with eigenvalue $\lambda_B$. 
\item $AB=\rho(ab)$ and $BA=\rho(ba)$ are parabolic or boundary elliptic and 
fix respectively $p_3$ and $p_4$.
\end{itemize}
\end{proposition}

\begin{proof}
  Define $A=\rho (a)$ and $B=\rho(b)$ using Proposition \ref{u-isom}: $A$ is
  the unique isometry with fixing $p_1$ with eigenvalue $\lambda_A$ and
  mapping $p_4$ to $p_3$, and $B$ is the unique isometry fixing
  $p_2$ with eigenvalue $\lambda_B$ and mapping $p_3$ to $p_4$. From this
  definition, we see that $AB$ fixes $p_3$ and $BA$ fixes $p_4$. 
It remains to check that the eigenvalue $\lambda_3$ of $AB$ associated to $p_3$ 
(which is the same as the eigenvalue $\lambda_4$ of $BA$ associated to $p_4$)
has unit modulus. From Lemma \ref{eigencross} we have
$$
\lambda_3 = \frac{\lambda_A\lambda_B}{\X(p_1,p_2,p_3,p_4)}.
$$
Since the tetrahedron is balanced, we have 
$\bigl|\X(p_1,p_2,p_3,p_4)\bigr|=1$ and the result follows.
\end{proof}

\begin{remark}
  The function mapping $(\tau,\lambda_A,\lambda_B)$ to the representation
  $\rho$ given by Proposition \ref{repEP} is not a bijection. Indeed
  in the case where one of $\rho(a)$, $\rho(b)$ or $\rho(c)$ is a
  complex reflections it does not have a unique fixed point, and so 
different ideal tetrahedra can give the same representation. 
\end{remark}

\subsection{A specific normalisation.}
We now give the parabolic representation of $F_2$ in ${\rm PU}(2,1)$ 
corresponding the balanced 
tetrahedron $\tau(\theta,\phi,\psi)$ given in Definition \ref{defitau}. 
This means that
\begin{equation}\label{balanced}
\bp_A=\begin{bmatrix}1\\0\\0\end{bmatrix}, \quad 
\bp_B=\begin{bmatrix}0\\0\\1\end{bmatrix}, \quad 
\bp_{AB}=\begin{bmatrix}
-e^{2i\theta} \\ \sqrt{2\cos(2\theta)}e^{i\theta-i\psi} \\ 1 \end{bmatrix}, \quad
\bp_{BA}=\begin{bmatrix}
-e^{-2i\phi} \\ \sqrt{2\cos(2\phi)}e^{-i\phi+i\psi} \\ 1 \end{bmatrix}
\end{equation}
where  
\begin{eqnarray*}
2\theta & = &\A(p_B,p_A,p_{AB})\in [-\pi/2,\pi/2]\\
2\phi & = &\A(p_A,p_B,p_{BA})\in [-\pi/2,\pi/2]\\
4\psi & = &\arg\left(\mathbb{B}(p_A,p_B,p_{AB},p_{BA})\right)\in[0,2\pi).
\end{eqnarray*}
%Note moreover that $\X(p_A,p_B,p_{AB},p_{BA})=e^{-2i(\theta+\phi)}$
%Suppose that $A$ and $B$ are the generators of a parabolic
%representation given given in Proposition \ref{repEP}. 
\noindent Writing $c_1=\sqrt{2\cos(2\theta)}$ and $c_2=\sqrt{2\cos(2\phi)}$, 
the matrices  $A$ and $B$ in ${\rm SU}(2,1)$ giving the parabolic 
representation are
\begin{eqnarray}
A & = & \begin{bmatrix}
\lambda_A 
& -\overline{\lambda}_A^2c_1\,e^{-i\theta+i\psi}+\lambda_Ac_2\,e^{i\phi-i\psi}
& -\lambda_Ae^{2i\theta}-\lambda_Ae^{2i\phi}+\overline{\lambda}_A^2
c_1c_2\,e^{-i\theta-i\phi+2i\psi} \\
0 & \overline{\lambda}_A^2 
& \lambda_Ac_1\,e^{i\theta-i\psi}-\overline{\lambda}_A^2c_2\,e^{-i\phi+i\psi} \\
0 & 0 & \lambda_A
\end{bmatrix}, \label{eq-A} \\
B & = & \begin{bmatrix}
\lambda_B & 0 & 0 \\
\overline{\lambda}_B^2c_1\,e^{-i\theta-i\psi}-\lambda_Bc_2\,e^{i\phi+i\psi} 
& \overline{\lambda}_B^2 & 0 \\
-\lambda_Be^{2i\theta}-\lambda_Be^{2i\phi}
+\overline{\lambda}_B^2c_1c_2\, e^{-i\theta-i\phi-2i\psi}
& -\lambda_Bc_1\,e^{i\theta+i\psi}+\overline{\lambda}_B^2c_2\,e^{-i\phi-i\psi}
& \lambda_B \end{bmatrix}. \label{eq-B}
\end{eqnarray}

\section{Thrice punctured sphere groups with a three-fold symmetry.}

In this section we restrict our attention to the case where there is a 
three-fold symmetry of the parabolic representation 
$\rho(F_2)=\langle A,B\rangle$. 
Consider the eigenvalues $\lambda_A$, $\lambda_B$ and $\lambda_C$
of $A$, $B$ and $C=B^{-1}A^{-1}$ at $p_A$, $p_B$ and $p_{AB}$.
Specifically, we show that if these are equal then $\langle A,B\rangle$
is an index 3 subgroup of a $(3,3,\infty)$ group $\langle J_1,J_2\rangle$.
Moreover, this can be interpreted geometrically, for there is
a bijection between $(3,3,\infty)$ groups and balanced ideal tetrahedra.

We go on to give conditions under which further elements of this
group are pinched, that is they have become parabolic. In doing so,
we rule out the case where they are complex reflections. Therefore
pinching a single element is equivalent to satisfying a single real
algebraic equation (Proposition \ref{prop-class-trace}) this defines
a real hypersurface. Our main result is that for the $(3,3,\infty)$
group it is possible to simultaneously pinch $J_1J_2^{-1}$ and
$[J_1,J_2]$. Indeed there is a 1 parameter way of doing this.
This means that for the thrice punctured sphere group, it is
possible to pinch four conjugacy classes in addition to the
three boundary curves. 

This is in strong contrast to the classical case.
Every thrice punctured sphere groups in 
${\rm SL}(2,{\mathbb R})$ or ${\rm SL}(2,{\mathbb C})$ admits a
three-fold symmetry, that is, it is an index three subgroup of a
$(3,3,\infty)$ group. However, it is not possible to make any more 
elements of this group parabolic.

\subsection{Existence of a three-fold symmetry.}
\begin{definition}\label{def-J1-J2}
Consider a balanced tetrahedron with vertices $p_A$, $p_B$, $p_{AB}$
and $p_{BA}$, all lying in $\partial{\bf H}^2_\C$. We define the
following elements of ${\rm PU}(2,1)$ (see Figure \ref{order3}):
\begin{itemize}
\item $J_1$ is the order 3 isometry cyclically permuting 
$p_B$, $p_A$ and $p_{AB}$.
\item $J_2$ is the order 3 isometry cyclically permuting 
$p_A$, $p_B$ and $p_{BA}$. 
\end{itemize}
When these triples of points do not lie in a complex line, such an isometry 
is unique. 
\begin{figure}
\begin{center}
 \includegraphics[scale=0.6]{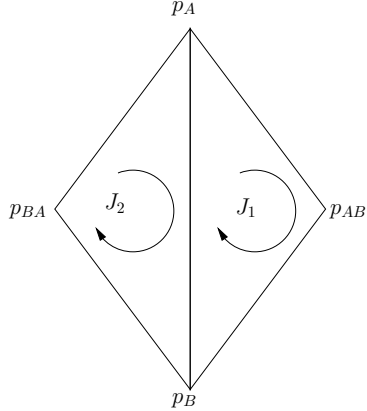}
\end{center}
\caption{Action of $J_1$ and $J_2$ on the fixed points of $A$, $B$, $AB$ and 
$BA$.\label{order3}}
\end{figure}
\end{definition}

Using the lifts of the vertices given in \eqref{balanced}, 
as matrices in ${\rm SU}(2,1)$ the maps $J_1$ and $J_2$ from
Definition \ref{def-J1-J2} are given by
\begin{eqnarray}
J_1 & = & \left[\begin{matrix}
e^{4i\theta/3} & \sqrt{2\cos(2\theta)}e^{i\theta/3+i\psi} & -e^{-2i\theta/3} \\
-\sqrt{2\cos(2\theta)}e^{i\theta/3-i\psi} & -e^{4i\theta/3} & 0 \\
-e^{-2i\theta/3} & 0 & 0 \end{matrix}\right], \label{eq-J1} \\
J_2 & = & \left[\begin{matrix}
0 & 0 & -e^{-2i\phi/3} \\
0 & -e^{4i\phi/3} & \sqrt{2\cos(2\phi)}e^{i\phi/3+i\psi} \\
-e^{-2i\phi/3} & -\sqrt{2\cos(2\phi)}e^{i\phi/3-i\psi} & e^{4i\phi/3}
\end{matrix}\right]. \label{eq-J2}
\end{eqnarray}
The ambiguity in the lift from ${\rm PU}(2,1)$ to ${\rm SU}(2,1)$
is precisely the same as the choice of cube root of $e^{i\theta}$ and
$e^{i\phi}$. Then we immediately have
\begin{eqnarray}
J_1^{-1} & = & \left[\begin{matrix}
0 & 0 & -e^{2i\theta/3} \\
0 & -e^{-4i\theta/3} & \sqrt{2\cos(2\theta)}e^{-i\theta/3-i\psi} \\
-e^{2i\theta/3} & -\sqrt{2\cos(2\theta)}e^{-i\theta/3+i\psi} & e^{-4i\theta/3}
\end{matrix}\right], \label{eq-J1i} %%%\\
% % % J_2^{-1} & = & \left[\begin{matrix}
% % % e^{-4i\phi/3} & \sqrt{2\cos(2\phi)}e^{-i\phi/3-i\psi} & -e^{2i\phi/3} \\
% % % -\sqrt{2\cos(2\phi)}e^{-i\phi/3+i\psi} & -e^{-4i\phi/3} & 0 \\
% % % -e^{2i\phi/3} & 0 & 0 \end{matrix}\right]. \label{eq-J2i}
\end{eqnarray}

\begin{theorem}\label{thm-symmetric-groups}
Let $\rho : F_2\longrightarrow{\rm PU}(2,1)$ be a representation so that
$A=\rho(a)$, $B=\rho(b)$ and $AB=\rho(c^{-1})$ are all parabolic
and let $p_A$, $p_B$ and $p_{AB}$ be their fixed points. 
Let $J_1$ be the order three map cyclically permuting $p_B$, $p_A$ and $p_{AB}$.
Let $p_{BA}$ be the fixed point of $BA$ and let $J_2$ be the order three 
map cyclically permuting $p_A$, $p_B$ and $p_{BA}$. 
Then the following are equivalent:
\begin{itemize}
\item[(i)] $A=J_1J_2$ and $B=J_2J_1$.
\item[(ii)] $\lambda_A$ and $\lambda_B$ are equal to the same cube root 
of the cross ratio $\X(p_A,p_B,p_{AB},p_{BA})$.
\end{itemize}
\end{theorem}

\begin{proof}
Suppose that $A=J_1J_2$ and $B=J_2J_1$. Then 
$
(AB)^{-1}=J_1^{-1}J_2J_1^{-1}=J_1^{-1}BJ_1=J_1AJ_1^{-1}.
$
Therefore, $A$, $B$ and $C=B^{-1}A^{-1}$ are all conjugate, and so 
$\lambda_A=\lambda_B=\lambda_C$. Using Corollary \ref{3spherecrossratio} 
they must be all equal to the same cube root of $\X(p_A,p_B,p_{AB},p_{BA})$.

Conversely, assume that $\lambda_A=\lambda_B$
and $\lambda_A^3=\X(p_A,p_B,p_{AB},p_{BA})=e^{-2i\theta-2i\phi}$, and
consider the two isometries
$$
A'=J_1J_2\quad\mbox{ and }\quad B'=J_2J_1.
$$ 
Clearly $A'$ and $B'$ are conjugate. Moreover, they are also conjugate to
$$
C'=(A'B')^{-1}=J_1^{-1}J_2J_1^{-1}=J_1A'J_1^{-1}.
$$ 
From the definition of $J_1$ and $J_2$ we see that 
$A'(p_A)=J_1J_2(p_A)=J_1(p_B)=p_A$ so $A'$ fixes $p_A$.
Similarly $B'$ fixes $p_B$, $A'B'$ fixes $p_{AB}$ and $B'A'$ fixes $p_{BA}$. 
As a consequence of Lemma \ref{eigencross}, we see that the eigenvalues 
$\lambda_{A'}$, $\lambda_{B'}$, $\lambda_{C'}$ satisfy 
$$
\X(p_{A},p_B,p_{AB},p_{BA})
=\dfrac{1}{\overline\lambda_{A'}\overline\lambda_{B'}\overline\lambda_{C'}}.
$$
As the cross ratio has unit modulus, it implies that the three eigenvalue 
have unit modulus. As they are equal (the three isometries are conjugate), 
they are all equal to the same cube root of $\X(p_{A},p_B,p_{AB},p_{BA})$
Using Proposition \ref{u-isom}, this implies that $A=A'$ and $B=B'$. 
\end{proof}

\medskip

The following proposition is a straightforward corollary.

\begin{corollary}
The following two conditions are equivalent.
\begin{enumerate}
\item The eigenvalue $\lambda$ of $J_1J_2$ associated with $p_1$ has 
unit modulus.
\item The tetrahedron $(p_1,p_2,p_3,p_4)$ is balanced.
\end{enumerate}
In this case,  $\X(p_1,p_2,p_3,p_4)=\lambda^3$.
\end{corollary}

Because $J_1$ and $J_2$ have order three, we see that
\begin{eqnarray}
% % AB & = & J_1J_2J_2J_1=J_1J_2^{-1}J_1^{-1}J_1^{-1}=J_1 A^{-1} J_1^{-1}, \\
% % BA & = & J_2J_1J_1J_2=J_2J_1^{-1}J_2^{-1}J_2^{-1}=J_2B^{-1}J_2^{-1}, \\
AB^{-1} & = & J_1J_2J_1^{-1}J_2^{-1}=[J_1,J_2], \label{eq-ABinv} \\
\lbrack A,B\rbrack & = & 
ABA^{-1}B^{-1}
=(J_1J_2)(J_2J_1)(J_2^{-1}J_1^{-1})(J_1^{-1}J_2^{-1})=(J_1J_2^{-1})^3.
\label{eq-comm-AB}
\end{eqnarray}

\subsection{Parameters \label{section-params}}
%From Theorem \ref{thm-main1} we see that there is a bijection between
%the set of ${\rm PU}(2,1)$-orbits of balanced tetrahedra and the
%conjugacy classes of groups $\Gamma<{\rm PU}(2,1)$ generated by $J_1$
%and $J_2$ with $J_1$, $J_2$ regular elliptic of order three and
%$J_1J_2$ parabolic. 
We have seen, Proposition \ref{prop-tet-params},
that a balanced tetrahedron with ideal vertices $p_1$, $p_2$, $p_3$ and $p_4$ 
is determined up to ${\rm PU}(2,1)$ equivalence by 
$$
2\theta=\A(p_2,p_1,p_3),\quad
2\phi=\A(p_1,p_2,p_4),\quad
4\psi={\rm arg}\bigl({\mathbb B}(p_1,p_2,p_3,p_4)\bigr).
$$
In the next sections we write certain traces in terms of these
parameters $\theta$, $\phi$, $\psi$. We will then obtain equations
in these variables that determine when certain words in the group
$\Gamma=\langle J_1,\,J_2\rangle$ are parabolic or unipotent.
It turns out that many of these computations become
easier if we switch to the following real variables.
\begin{equation}\label{eq-new-params}
x=4\sqrt{\cos(2\theta)\cos(2\phi)}\cos(2\psi), \quad
y=4\sqrt{\cos(2\theta)\cos(2\phi)}\sin(2\psi), \quad
z=4\cos(\theta-\phi).
\end{equation}
Recall that $\theta\in [-\pi/4,\pi/4]$, $\phi\in[-\pi/4,\pi/4]$ and 
$\psi\in[0,\pi/2]$.
Note that $z\ge 0$ with equality if and only if $\phi=-\theta=\pm\pi/4$
and $z\le 4$ with equality if and only if $\phi=\theta$. Furthermore, note that
\begin{eqnarray*}
2\cos^2(\theta-\phi) & = & 1+\cos(2\theta-2\phi) \\
& \ge & \cos(2\theta+2\phi)+\cos(2\theta-2\phi) \\
& = & 2\cos(2\theta)\cos(2\phi).
\end{eqnarray*}
This implies that $z^2\ge x^2+y^2$ with equality if and only if $\phi=-\theta$.
The latter inequality implies $-z\le x\le z$ and $-z\le y\le z$. Note for 
later use that in particular the condition $x=z$ implies that $\phi=-\theta$ and $\psi=0$.
The Jacobian associated to the change of variable \eqref{eq-new-params} is
${\mathcal J} =128\sin(2\theta+2\phi)\sin(\theta-\phi)$.
% % \begin{eqnarray*}
% % {\mathcal J} 
% % & = & 
% % {\rm det}\left(\begin{matrix}
% % -4\sin(2\theta)\sqrt{\dfrac{\cos(2\phi)}{\cos(2\theta)}}\cos(2\psi) &
% % -4\sin(2\theta)\sqrt{\dfrac{\cos(2\phi)}{\cos(2\theta)}}\sin(2\psi) &
% % -4\sin(\theta-\phi) \\
% % & & \\
% % -4\sin(2\phi)\sqrt{\dfrac{\cos(2\theta)}{\cos(2\phi)}}\cos(2\psi) &
% % -4\sin(2\phi)\sqrt{\dfrac{\cos(2\theta)}{\cos(2\phi)}}\sin(2\psi) &
% % 4\sin(\theta-\phi) \\
% % & & \\
% % -8\sqrt{\cos(2\theta)\cos(2\phi)}\sin(2\psi) &
% % 8\sqrt{\cos(2\theta)\cos(2\phi)}\cos(2\psi)  &
% % 0
% % \end{matrix}\right) \\
% % &&\\
% % & = & 128\sin(2\theta+2\phi)\sin(\theta-\phi).
% % \end{eqnarray*}
Therefore, this change of variables is a local diffeomorphism at all
points where $\theta\neq \pm \phi$. 
\subsection{Ruling out complex reflections.\label{sec-no-cx-refl}}
The goal of this section is to describe the isometry type of certain 
elements of the group $\la J_1,J_2\ra$, and show that they can not be complex 
reflections. More precisely, we are going to prove that if $J_1J_2$, 
$J_1J_2^{-1}$ or $[J_1,J_2]$ has a neutral fixed point, then it is either 
parabolic of the identity. We begin by studying the product $J_1J_2$. 
It is possible to find an expression for $A=J_1J_2$ 
and $B=J_2J_1$ by plugging $\lambda_A=\lambda_B=e^{-2i\theta/3-2i\phi/3}$ in 
\eqref{eq-A} and \eqref{eq-B}. This leads to
% To simplify the notation, write 
% $c_1=\sqrt{2\cos(2\theta)}$ and $c_2=\sqrt{2\cos(2\phi)}$. Then
% \begin{eqnarray}
% J_1J_2
% & = & e^{i\theta/3+i\phi/3}
% \left[\begin{matrix}
% e^{-i\theta-i\phi} &
% -e^{i\phi+i\psi}c_1
% +e^{-i\theta-i\psi}c_2 &
% -e^{i\theta-i\phi}+e^{2i\psi}c_1c_2
% -e^{-i\theta+i\phi} \\
% 0 & e^{i\theta+i\phi} &
% e^{-i\phi-i\psi}c_1
% -e^{i\theta+i\psi}c_2 \\
% 0 & 0 & e^{-i\theta-i\phi} \end{matrix}\right], \label{eq-mat-J1J2}\\
% J_1^{-1}J_2^{-1}
% & = & e^{-i\theta/3-i\phi/3}\left[\begin{matrix}
% e^{i\theta+i\phi} & 0 & 0 \\
% -e^{i\phi-i\psi}c_1
% +e^{-i\theta+i\psi}c_2 & e^{-i\theta-i\phi} & 0 \\
% -e^{i\theta-i\phi}+e^{2i\psi}c_1c_2
% -e^{-i\theta+i\phi} &
% e^{-i\phi+i\psi}c_1-e^{i\theta-i\psi}c_2 & e^{i\theta+i\phi}
% \end{matrix}\right], \label{eq-mat-J1iJ2i}\\
% \end{eqnarray}
\begin{equation} 
{\rm tr}(J_1J_2) = 2e^{-2i\theta/3-2i\phi/3}+e^{4i\theta/3+4i\phi/3}. \label{trJ1J2}
\end{equation}
In particular ${\rm tr}(J_1J_2)$ lies on the deltoid curve described in 
Section \ref{section-conj-class} 
(see figure \ref{parabcurve}), and we have to decide if $J_1J_2$ is parabolic, 
a complex reflection or the identity.

\begin{proposition} \label{prop-A-parab-identity}
The map $J_1J_2$ is always parabolic unless
$p_{AB}=p_{BA}$, in which case it is the identity. In particular, it
cannot be a non trivial reflection.  
\end{proposition}
\begin{proof} 
Using Proposition \ref{u-isom} with $p=p_A$, $q=p_{BA}$ and
$r=p_{AB}$, we see that $J_1J_2$ is a complex reflection if and only if its 
eigenvalue $\lambda_A$ associated to $p_A$ satisfies 
$\lambda_A^3=-\exp\bigl(2i\A(p_A,p_{BA},p_{AB})\bigr)$.  But
we know from Corollary \ref{3spherecrossratio} and the three-fold symmetry that 
$$
\lambda_A^3=\X(p_A,p_B,p_{AB},p_{BA}).
$$ 
Combining these two relations, taking argument on both sides, and
using part 3 of Lemma \ref{eigencross},  we obtain that
\begin{equation}\label{eq-argu-cr}
\A(p_A,p_B,p_{AB})-\A(p_A,p_B,p_{BA})=\pi+2\A(p_A,p_{BA},p_{AB})\mod 2\pi.
\end{equation}
On the other hand, the cocycle relation of the Cartan invariant
(Corollary 7.1.12 of \cite{Goldbook}) gives us
\begin{equation}\label{eq-cocycle}
\A(p_A,p_B,p_{AB})-\A(p_A,p_B,p_{BA})+\A(p_A,p_{AB},p_{BA})-\A(p_B,p_{AB},p_{BA})=0.
\end{equation}
Summing equations \eqref{eq-argu-cr} and \eqref{eq-cocycle} gives
$$
\A(p_A,p_{AB},p_{BA})+\A(p_B,p_{AB},p_{BA})=\pi \mod 2\pi.
$$
As these two Cartan invariants belong to $[-\pi/2,\pi/2]$ (see Chapter
7 of \cite{Goldbook}), they must be either both equal to $\pi/2$ or
both equal to $-\pi/2$. This means that the four points $p_A$, $p_B$,
$p_{AB}$ and $p_{BA}$ belongs to a common complex line $L$ (Corollary
7.1.13 of \cite{Goldbook}). Moreover the fact that
$\A(p_A,p_{AB},p_{BA})$ and $\A(p_B,p_{AB},p_{BA})$ have the same sign
means that $p_A$ and $p_B$ lie on the same side of the geodesic
connecting $p_{AB}$ and $p_{BA}$. As the tetrahedron
$(p_A,p_B,p_{AB},p_{BA})$ is balanced, $p_{AB}$ and $p_{BA}$
orthogonally project onto the same point of the geodesic
$(p_Ap_B)$. This is only possible when $p_{AB}=p_{BA}$. This implies that 
$J_2=J_1^{-1}$.
\end{proof} 

In $(\theta,\phi,\psi)$-coordinates, it is straightforward to check that
$p_{AB}=p_{BA}$ if and only if $\psi=0$ and $\theta=-\phi$. 
Therefore we see that $J_1J_2$ can only be a complex reflection when 
$\phi=-\theta$ and $\psi=0$. Plugging these values in \eqref{eq-J2} 
and \eqref{eq-J1i}, we see that this implies $J_2=J_1^{-1}$.

\begin{remark}\label{rem-x=z}
Note that the relation  
$2\cos(\theta-\phi)-2\sqrt{\cos(2\theta)\cos(2\phi)}\cos(2\psi)=0$ rewrites in 
$(x,y,z)$ coordinates as $x=z$. The previous discussion shows thus that 
$x=z$ implies that $J_1J_2$ is the identity.
\end{remark}

\begin{corollary}\label{cor-J1J2i-comm-not para}
The maps $J_1J_2^{-1}$ and $[J_1J_2]$ are never complex reflections.
\end{corollary}

\begin{proof}
In Proposition \ref{prop-A-parab-identity} the only facts we have used 
about $J_1$ and $J_2$ are that $J_1$ and $J_2$ have order three and
their product has a neutral fixed point on the boundary. 
By changing $J_2$ to $J_2^{-1}$ or $J_2J_1J_2^{-1}$ respectively,
we see that if $J_1J_2^{-1}$ or $[J_1,J_2]$ has a neutral fixed
point on the boundary then it is parabolic or the identity.
\end{proof} 
The following result is a straightforward consequence of the previous Proposition  \ref{prop-A-parab-identity} 
(note that a $(3,3,\infty)$-group is a group generated by two order three elements of which product is parabolic).
\begin{theorem}\label{thm-main1}
There is a bijection between the set of ${\rm PU}(2,1)$-orbits of non 
degenerate balanced tetrahedra, and the set of ${\rm PU}(2,1)$-conjugacy 
classes of $(3,3,\infty)$-groups in ${\rm PU}(2,1)$.
\end{theorem}
Here by non degenerate, we mean the the four vertices of the tetrahedron are 
pairwise distinct.
\begin{remark}\label{rem-J1J2inv}
It follows from Corollary \ref{cor-J1J2i-comm-not para} that whenever 
$f(\tr(J_1J_2^{-1}))=0$, then  
$J_1J_2^{-1}$ is parabolic or the identity. For later use, we compute 
$f(\tr(J_1J_2^{-1}))$. First, a 
simple computation shows
\begin{equation}J_1J_2^{-1} = e^{i\theta/3-i\phi/3}\left[\begin{matrix}
e^{i\theta-i\phi}-c_1c_2e^{2i\psi}+e^{-i\theta+i\phi} &
-c_1e^{-i\phi+i\psi}+c_2e^{i\theta-i\psi} & 
-e^{i\theta+i\phi} \\
-c_1e^{-i\phi-i\psi}+c_2e^{i\theta+i\psi} &
e^{i\theta-i\phi}-c_1c_2e^{-2i\psi} & 
c_1e^{i\phi-i\psi} \\
-e^{-i\theta-i\phi} & -c_2e^{-i\theta-i\psi} & e^{-i\theta+i\phi}
\end{matrix}\right]\label{eq-mat-J1J2i}.
\end{equation}
Therefore
\begin{eqnarray}
{\rm tr}(J_1J_2^{-1}) & = & 
e^{i\theta/3-i\phi/3}\Bigl(4\cos(\theta-\phi)
-4\sqrt{\cos(2\theta)\cos(2\phi)}\cos(2\psi)\Bigr)\nonumber\\
& = & e^{i\theta/3-i\phi/3}\left(z-x\right)\label{trJ1J2m}
\end{eqnarray}
Plugging this value into Proposition \ref{prop-class-trace}, we obtain 
after rearranging that 
\begin{equation}\label{eq-tr-J1J2m-delto}f\left({\rm tr}(J_1J_2^{-1})\right)=(x-z)^2\left(x^2-z^2+18\right)-27.
\end{equation}
The hypersurface defined by this equation is shown (in $(\theta,\phi ,\psi)$-coordinates) in black in 
Figure \ref{figureJ1J2parab}. It is interesting to note that if $J_1J_2^{-1}$ is parabolic, then the 
above quantity must be non zero and thus $x-z\neq 0$. This implies that when $J_1J_2^{-1}$ is parabolic, 
so is $J_1J_2$.
\end{remark}

\subsection{Super-pinching}
In this section we show that it is possible to have a one parameter family
of representations of $F_2$ to ${\rm SU}(2,1)$ with seven primitive 
conjugacy classes of parabolic map. Because we also impose 3-fold 
symmetry, this is the same as saying that we have a one parameter 
family of representations of $\Z_3 * \Z_3$ with three primitive 
parabolic conjugacy classes. 
\begin{theorem}\label{thm-superpinching}
There is a one parameter family of groups generated by $J_1$ and $J_2$
in ${\rm SU}(2,1)$ with the following properties:
\begin{itemize}
\item $J_1$ and $J_2$ are both elliptic maps of order 3;
\item $J_1J_2$, $J_1J_2^{-1}$ and $[J_1,J_2]$ are all parabolic.
\end{itemize}
\end{theorem}
Passing to the subgroup generated by $A=J_1J_2$ and $B=J_2J_1$, this implies
\begin{corollary}\label{cor-superpinching}
There is a one parameter family of groups generated by $A$ and $B$
in ${\rm SU}(2,1)$ with $A$, $B$, $AB$, $AB^{-1}$, $AB^2$, $A^2B$ and
$[A,B]$ all parabolic.
\end{corollary}

\begin{proof}
In the groups from Theorem \ref{thm-superpinching}
we write $A=J_1J_2$, $B=J_2J_1$, leading to $AB=J_1A^{-1}J_1^{-1}$, so these
maps are all parabolic. Furthermore, using \eqref{eq-ABinv} we see that
$AB^{-1}=[J_1,J_2]$ is parabolic, and so is $BAB=J_1^{-1}AB^{-1}J_1$ and
$A^2B=J_1BA^{-1}J_1^{-1}$. Finally, using \eqref{eq-comm-AB}
we see $[A,B]=(J_1J_2^{-1})^3$ is also parabolic.
\end{proof}

\medskip

\begin{lemma}\label{lem-comm}
 In $(x,y,z)$-coordinates the trace for the commutator $[J_1,J_2]$ is given by
\begin{equation}\label{eq-tr-comm}
 {\rm tr}[J_1,J_2] = 3+\dfrac{(x-z)(3x-z)+y^2+2i(x-z)y}{4}
\end{equation}
\end{lemma}

\begin{proof}
By direct computation from the expressions for $J_1J_2$ and $J_1^{-1}J_2^{-1}$
above we find:
\begin{eqnarray*}
{\rm tr}[J_1,J_2] 
& = & 5 + 8\cos(2\theta)\cos(2\phi)+2\cos(2\theta-2\phi) \\
&& \quad -12\sqrt{\cos(2\theta)\cos(2\phi)}\cos(\theta-\phi)e^{2i\psi}
-4\sqrt{\cos(2\theta)\cos(2\phi)}\cos(\theta-\phi)e^{-2i\psi} \\
&& \quad +4\cos(2\theta)\cos(2\phi)e^{4i\psi}.
\end{eqnarray*}
Simplifying and changing variables gives the result.
\end{proof}

\begin{proof}(Theorem \ref{thm-superpinching}.)
We again use the change of variables
\eqref{eq-new-params}, namely
$$
x=4\sqrt{\cos(2\theta)\cos(2\phi)}\cos(2\psi), \quad
y=4\sqrt{\cos(2\theta)\cos(2\phi)}\sin(2\psi), \quad
z=4\cos(\theta-\phi).
$$
By construction, we know that $J_1$ and $J_2$ are both regular elliptic maps
of order three and that $J_1J_2$ is parabolic or
a complex reflection. Moreover, we know from Remark \ref{rem-J1J2inv} that if 
$J_1J_2^{-1}$ is parabolic, so is $J_1J_2$. Let us assume that both are 
parabolic and consider the commutator $[J_1J_2]$. 
Rewriting condition \eqref{eq-tr-J1J2m-delto}, we obtain 
\begin{equation}\label{eq-J1J2inv-para}
2z(x-z)=\frac{27-(x-z)^4-18(x-z)^2}{(x-z)^2}.
\end{equation}
Substituting this identity into the expression \eqref{eq-tr-comm} for 
${\rm tr}[J_1,J_2]$ and simplifying, yields:
$$
{\rm tr}[J_1,J_2]=\frac{2(x-z)^4-6(x-z)^2+27+(x-z)^2y^2
+2i(x-z)^3y}{4(x-z)^2}.
$$
Our goal will be to substitute this expression into 
Proposition \ref{prop-class-trace}. Specifically, 
by Corollary \ref{cor-J1J2i-comm-not para}, if 
$f\bigl({\rm tr}[J_1,J_2]\bigr)=0$ then $[J_1,J_2]$ will be parabolic. 
Such solutions will be exactly the groups we are looking for.
To simplify the expressions as much as possible, we make a further
change of variables, namely we write $X=(x-z)^2$ and $Y=(x-z)y$. 
With respect to these new variables, we have:
$$
{\rm tr}[J_1,J_2]=\frac{2X^2-6X+27+Y^2+2iXY}{4X}.
$$
Plugging this into Proposition \ref{prop-class-trace} and
simplifying, we find that
$$
256X^4\,f\bigl({\rm tr}[J_1,J_2]\bigr)=P(X,Y)
$$ 
where
\begin{eqnarray*}
P(X,Y) & = & 
Y^8+4(4X^2-14X+27)Y^6+6(12X^4-8X^3+360X^2-756X+729)Y^4 \\
&& \quad 
+4(16X^6-24X^5+1404X^4-4536X^3+20412X^2-30618X+19683)Y^2 \\ 
&& \quad 
+(2X^2-2X+27)(2X^2-18X+27)^3
%+16X^8-448X^7+5184X^6-33696X^5+145800X^4 \\
%&& \quad
%-454896X^3+944784X^2-1102248X+531441 \\
\end{eqnarray*}
Therefore, in order to find groups where $[J_1,J_2]$ is parabolic or a complex
reflection, we must identify those values of $X$ for which there exists $Y$
with $P(X,Y)=0$. It is clear that for a given value of $X$ and 
large enough values of $Y$ we must have $P(X,Y)>0$. Therefore for each $X$ 
such that $P(X,0)<0$ there exists $Y$ such that $P(X,Y)=0$. But 
$P(X,0)=(2X^2-2X+27)(2X^2-18X+27)^3$, and $(2X^2-2X+27)>0$ on $\R$. 
It follows from this fact that $P(X,0)\le 0$ if and only if
%Now we have
%\begin{eqnarray*}
%P(X,0) & = & 16X^8-448X^7+5184X^6-33696X^5+145800X^4 \\
%&& \quad
%-454896X^3+944784X^2-1102248X+531441 \\
%& = & (2X^2-2X+27)(2X^2-18X+27)^3.
%\end{eqnarray*}
%It is not hard to see that $2X^2-2X+27>0$ for all real $X$. However,
%if
\begin{equation}\label{eq-X-range}
\frac{9-3\sqrt{3}}{2}\le X\le \frac{9+3\sqrt{3}}{2}
\end{equation}
%then $2X^2-18X+27\le 0$ and so $P(X,0)\le 0$. 
Therefore, for this range of $X$ there exists a $Y$ with $P(X,Y)=0$. 
In Figure \ref{fig-comm-para} we illustrate the locus $P(X,Y)=0$ in this range
\end{proof}
\begin{figure}  
\begin{center}
\includegraphics[scale=0.4]{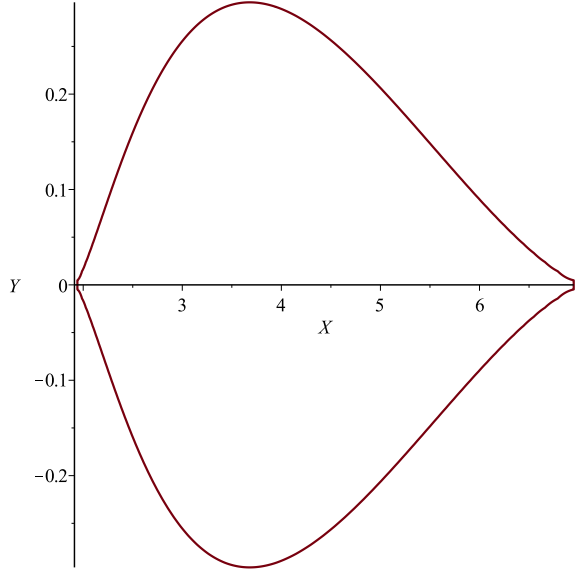}
\caption{The locus $P(X,Y)=0$ in the range
$\bigl(9-3\sqrt{3}\bigr)/2\le X\le \bigl(9+3\sqrt{3}\bigr)/2$. 
\label{fig-comm-para}}
\end{center}
\end{figure}
\begin{remark}
 Computing the resultant of $P(X,Y)$ and $\partial P/\partial Y$ with respect 
to $X$, it is 
possible to verify that the curve depicted on Figure \ref{fig-comm-para} is 
in fact the full 
zero locus of $P$ on $\R^+\times\R$. This can be done easily using 
computation software such as MAPLE. This 
indicates that the set of classes of groups $\la J_1,J_2\ra$ having these 
property is reduced to this 
(topological) circle. 
\end{remark}

%Note that when we have values $X$ and $Y$ for which $P(X,Y)=0$
%we can use $X=(x-z)^2$ and \eqref{eq-J1J2inv-para} to solve
%for $x$ and $z$. Having done so, we can use $Y=(x-z)y$ to solve for
%$y$. Finally, this enables us to obtain $\theta$, $\phi$ and
%$\psi$ using the identities \eqref{eq-new-params}.

%We illustrate this by considering the points where $P(X,0)=0$.
%Note that condition is equivalent to $2X^2-18X+27=0$ and $Y=0$,
%and so it corresponds
%to ${\rm tr}[J_1,J_2]=3$. In this case, we have
%$X=(9\pm3\sqrt{3})/2$ and $Y=0$. In other words,
%$(x-z)^2=(9\pm3\sqrt{3})/2$ and so
%$$
%z(x-z)=\frac{27-(x-z)^4-18(x-z)^2}{2(x-z)^2}
%=\frac{-3(9\pm3\sqrt{3})}{4}.
%$$
%This leads to the solution
%$$
%x=\frac{\sqrt{9\pm3\sqrt{3}}}{2\sqrt{2}},\quad y=0,\quad
%z=\frac{3\sqrt{9\pm3\sqrt{3}}}{2\sqrt{2}}.
%$$
%\end{proof}

\section{Discreteness}

So far we have not discussed discreteness. However, there are certain
subfamilies in our parameter space which have been studied before, and
where the range of discreteness is known. We discuss these case by case.

\subsection{Finite: $\theta=-\phi$, $\psi=0$.}

This is a simple case. It is easy to see that they imply $p_{AB}=p_{BA}$ 
and hence $J_2=J_1^{-1}$. In this case, the group has collapsed to a finite
group. Therefore, though discrete, this group is far from being
faithful.

\subsection{Ideal triangle groups: $\theta=-\phi$, $\psi=\pi/2$.}

The condition $\theta=-\phi$ implies that $J_1J_2$ is unipotent.
Furthermore, consider $I_0$, the complex reflection of order 2 
in the complex line spanned by $\infty=p_A$ and $o=p_B$. That is
$$
I_0=\left[\begin{matrix} -1 & 0 & 0 \\ 0 & 1 & 0 \\ 0 & 0 & -1
\end{matrix}\right]
$$
Observe that, as well as fixing $p_A$ and $p_B$, the involution
$I_0$ swaps $p_{AB}$ and $p_{BA}$. 

Using the definitions of $J_1$ and $J_2$,
this immediately implies $J_2=I_0J_1^{-1}I_0$. Writing $I_1=J_1I_0J_1^{-1}$ and
$I_2=J_1^{-1}I_0J_1$ we see that $J_1J_2=I_1I_0$, $J_2J_1=I_0I_2$
and $J_1^{-1}J_2J_1^{-1}=I_2I_1$ are all unipotent. Therefore these
groups are complex hyperbolic ideal triangle groups, as studied
by Goldman and Parker \cite{GP} and by Schwartz \cite{S1,S2,S3}. 
Schwartz's theorem is that such a group
is discrete provided $(I_1I_2I_0)^2=(J_1J_2^{-1})^3$ is not elliptic.
We have
$$
{\rm tr}(J_1J_2^{-1})=8\cos(2\theta)e^{2i\theta/3}.
$$
It is straightforward to check when the right hand side lies
outside the deltoid. Therefore we get the following reformulation of 
Schwartz's result:

\begin{theorem}\label{thm-itg}[Schwartz]
If $\theta=-\phi$, $\psi=\pi/2$ the group $\langle J_1,J_2\rangle$
is discrete and isomorphic to $\Z_3\star\Z_3$ if and only if
$$
\cos(2\theta)\ge \frac{\sqrt{3}}{8\sqrt{2}}.
$$
Moreover, for the value of $\theta$ where equality is attained,
the map $J_1J_2^{-1}$ is parabolic 
\end{theorem}

\subsection{Modular group deformations  1: $\theta=\phi$, $\psi=0$.}

Let $I_0$ be the following complex reflection in a complex line that swaps
$\infty=p_A$ and $o=p_B$:
$$
I_0=\left[\begin{matrix} 0 & 0 & 1 \\ 0 & -1 & 0 \\ 1 & 0 & 0
\end{matrix}\right].
$$
It is not hard to see that, as well as swapping $p_A$ and $p_B$, the
involution $I_0$ swaps $p_{AB}$ and $p_{BA}$. Thus we have $J_2=I_0J_1I_0$. 
Hence $J_1J_2=(J_1I_0)^2$. This means that $J_1I_0$ is also
parabolic. Since $I_0$ is a complex reflection fixing a complex line, 
these groups belong to the family of representations 
of the modular group considered by Falbel and Parker \cite{FP}.
Their main result, Theorem 1.2 of \cite{FP} is that such groups are 
discrete and faithful provided $J_1I_0J_1^{-1}I_0=J_1J_2^{-1}$ is not elliptic.
we have
$$
{\rm tr}(J_1J_2^{-1})=4-4\cos(2\theta).
$$
Therefore we can restate their result as:

\begin{theorem}\label{thm-mod1}[Falbel-Parker]
If $\theta=\phi$, $\psi=0$ the group $\langle J_1,J_2\rangle$
is discrete and isomorphic to $\Z_3\star\Z_3$ if and only if
$$
\cos(2\theta)\le \frac{1}{4}.
$$
Moreover, for the value of $\theta$ where equality is attained,
the map $J_1J_2^{-1}$ is parabolic 
\end{theorem}

\subsection{Modular group deformations  2: $\theta=\phi$, $\psi=\pi/2$.}

Now we take $I_0$ to be a complex reflection in a point that swaps 
$\infty=p_A$ and $o=p_B$. Namely:
$$
I_0=\left[\begin{matrix} 0 & 0 & -1 \\ 0 & -1 & 0 \\ -1 & 0 & 0 
\end{matrix}\right].
$$
Once gain, $I_0$ swaps $p_{AB}$ and $p_{BA}$ and so $J_2=I_0J_1I_0$
and $J_1I_0$ is parabolic. But since $I_0$ now fixes just a point,
we are in the family of representations of the modular group
considered by Falbel and Koseleff \cite{FK} and by
Gusevskii and Parker \cite{GuP}. The main result of
these papers is that such groups are discrete and faithful
for all values of $\theta$. We can restate this as:

\begin{theorem}\label{thm-mod2}[Falbel-Koseleff, Gusevskii-Parker]
If $\theta=\phi$, $\psi=\pi/2$ the group $\langle J_1,J_2\rangle$
is discrete and isomorphic to $\Z_3\star\Z_3$ for all 
$\theta\in [-\pi/4,\pi/4]$.
\end{theorem}

\subsection{Bending: $\theta=\phi=0$.}
We now consider the case where $\theta=\phi=0$ but $\psi$ is allowed
to vary. Since $0=2\theta=\A(p_B,p_A,p_{AB})$ and
$0=2\phi=\A(p_A,p_B,p_{BA})$ then the triples $(p_B,p_A,p_{AB})$
and $(p_A,p_P,p_{BA})$ each lie on an $\R$-circle. These
are the bending deformations of $\R$-Fuchsian groups constructed by Will 
in \cite{W1,W2}. The main result of \cite{W2}, which holds for any cusped 
surface group, is that these groups obtained by bending are discrete for a range 
of values of $\psi\in[0,\pi/4]$. Recently, these results have been extended in 
the case of the 3-punctured sphere by Parker and Will in \cite{ParkWill}. The main 
result of the latter paper comprises the fact that these groups are discrete and 
isomorphic to $F_2$ whenever $J_1J_2^{-1}$ is not elliptic. In the case where 
$\theta=\phi=0$, we have
$$
{\rm tr}(J_1J_2^{-1})=8\sin^2(\psi).
$$
The main result of \cite{ParkWill} implies thus the following:
\begin{theorem}\label{thm-bend}[Will, Parker-Will]
If $\theta=\phi=0$ the group $\langle J_1,J_2\rangle$
is discrete and isomorphic to $\Z_3\star\Z_3$ if and only if
$$
\sin(\psi)\ge \sqrt{\frac{3}{8}}.
$$
Moreover, for the value of $\psi$ where equality is attained,
the map $J_1J_2^{-1}$ is parabolic. 
\end{theorem}
Note that $\pi/4\sim 0.659$ and $\arcsin(\sqrt{3/8})\sim0.784$

\begin{figure}
\begin{center}
\includegraphics[scale=0.5]{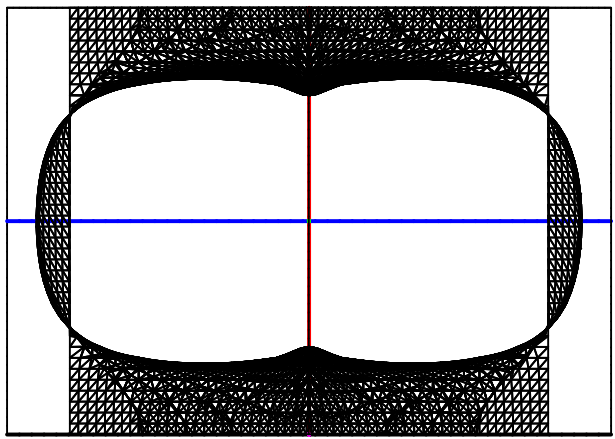} \\
\includegraphics[scale=0.5]{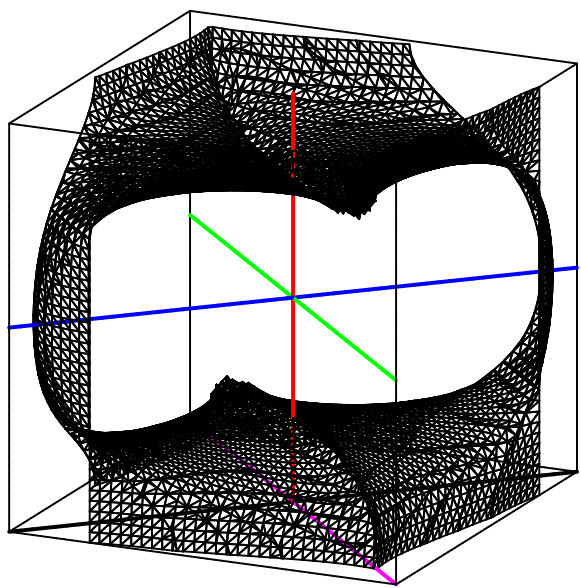}
\end{center}
\caption{\label{figureJ1J2parab} Two views of the parabolicity locus of 
$J_1J_2^{-1}$ and the special families. 
The colours are as follows : the black surface is the locus where $J_1J_2^{-1}$ 
is parabolic, the vertical red segment 
is the bending family, the black segment correspond to finite groups, the blue 
segment is the ideal triangle group case, the green 
and magenta segments are the two families corresponding to representations of 
the modular group.}
\end{figure}

\begin{figure}
\begin{center}
\begin{tabular}{cc}
 \includegraphics[scale=0.25]{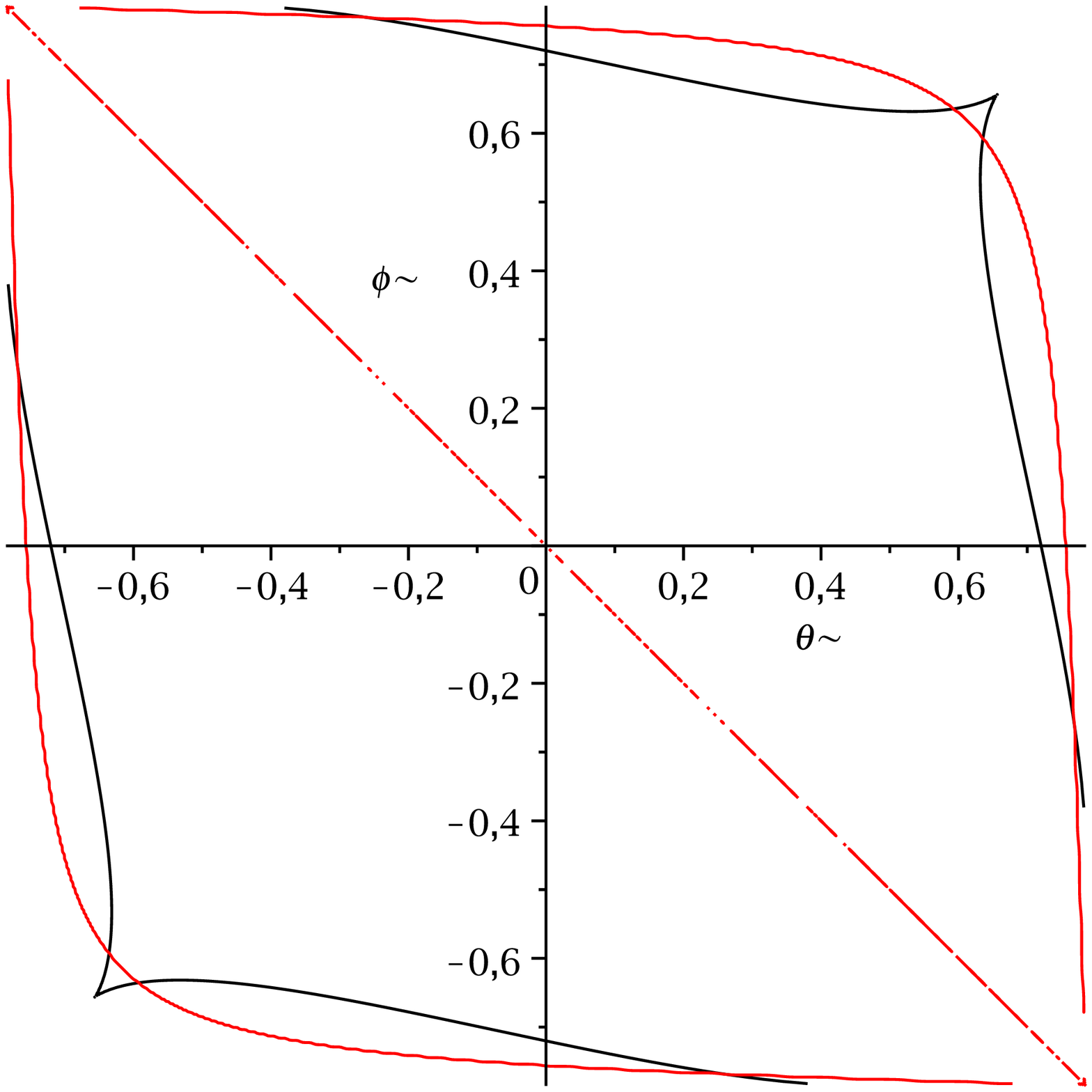} & 
\includegraphics[scale=0.25]{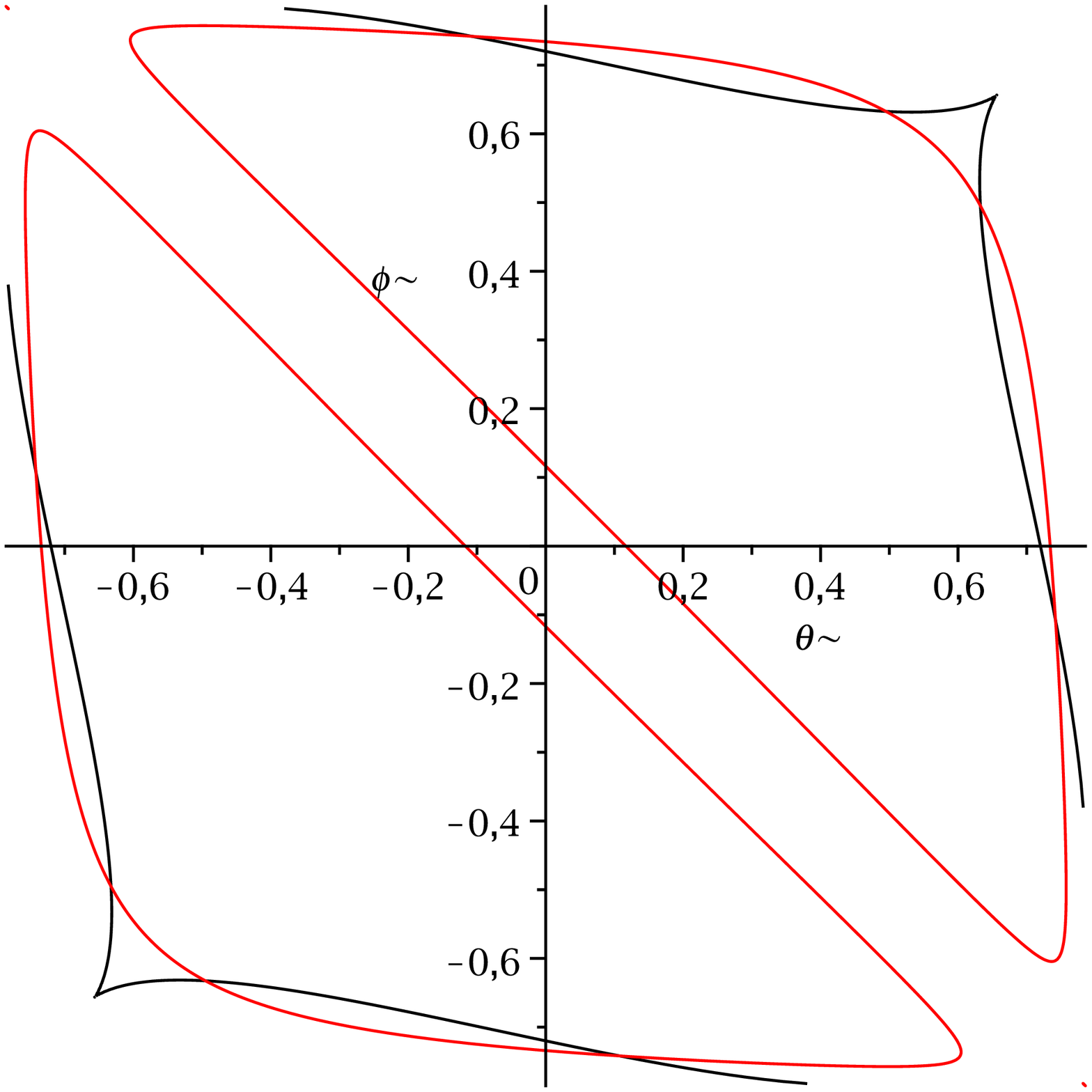}\\
&\\
$\psi=0$ & $\psi=0.02$\\
&\\
 \includegraphics[scale=0.25]{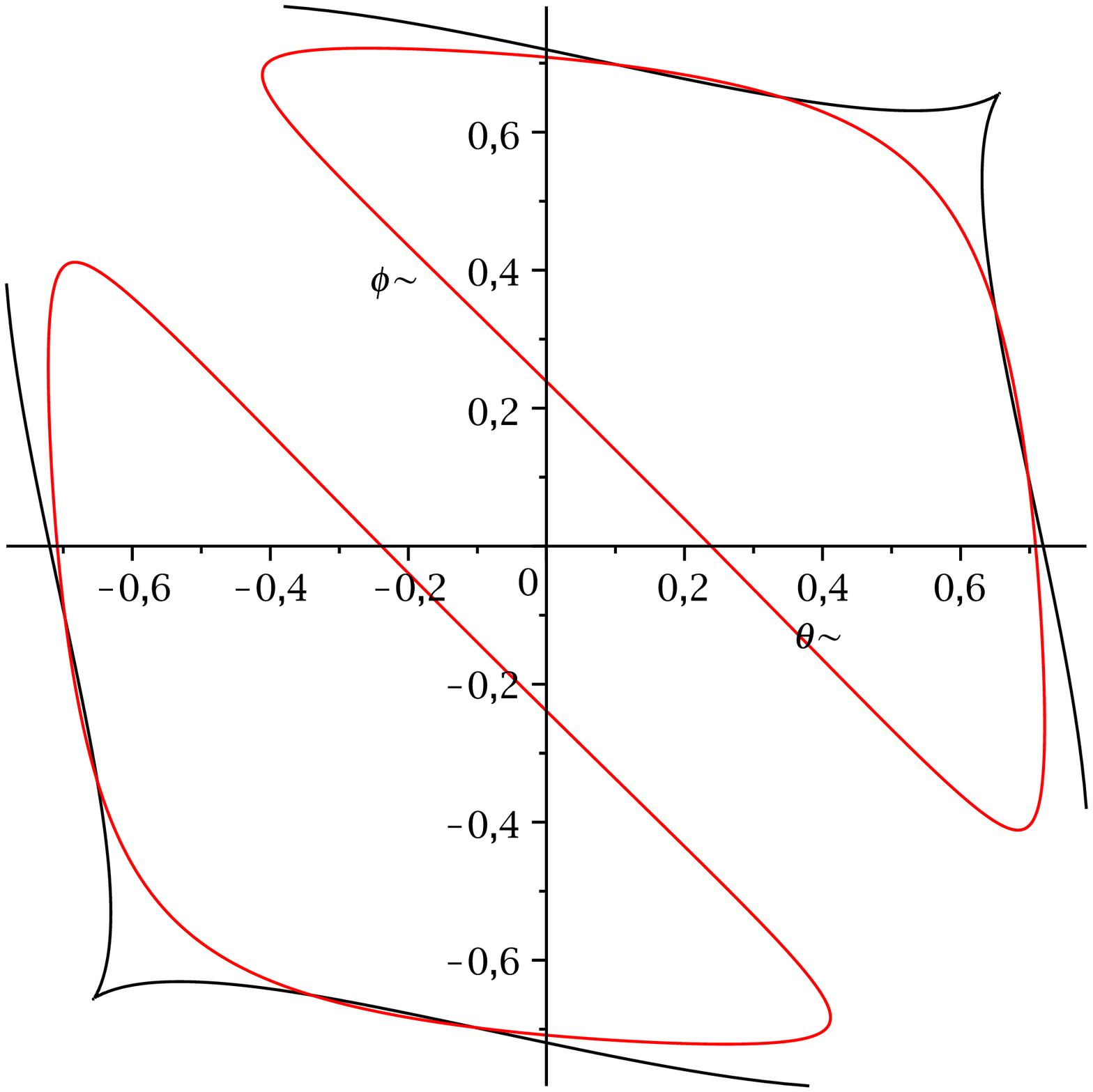} & 
\includegraphics[scale=0.25]{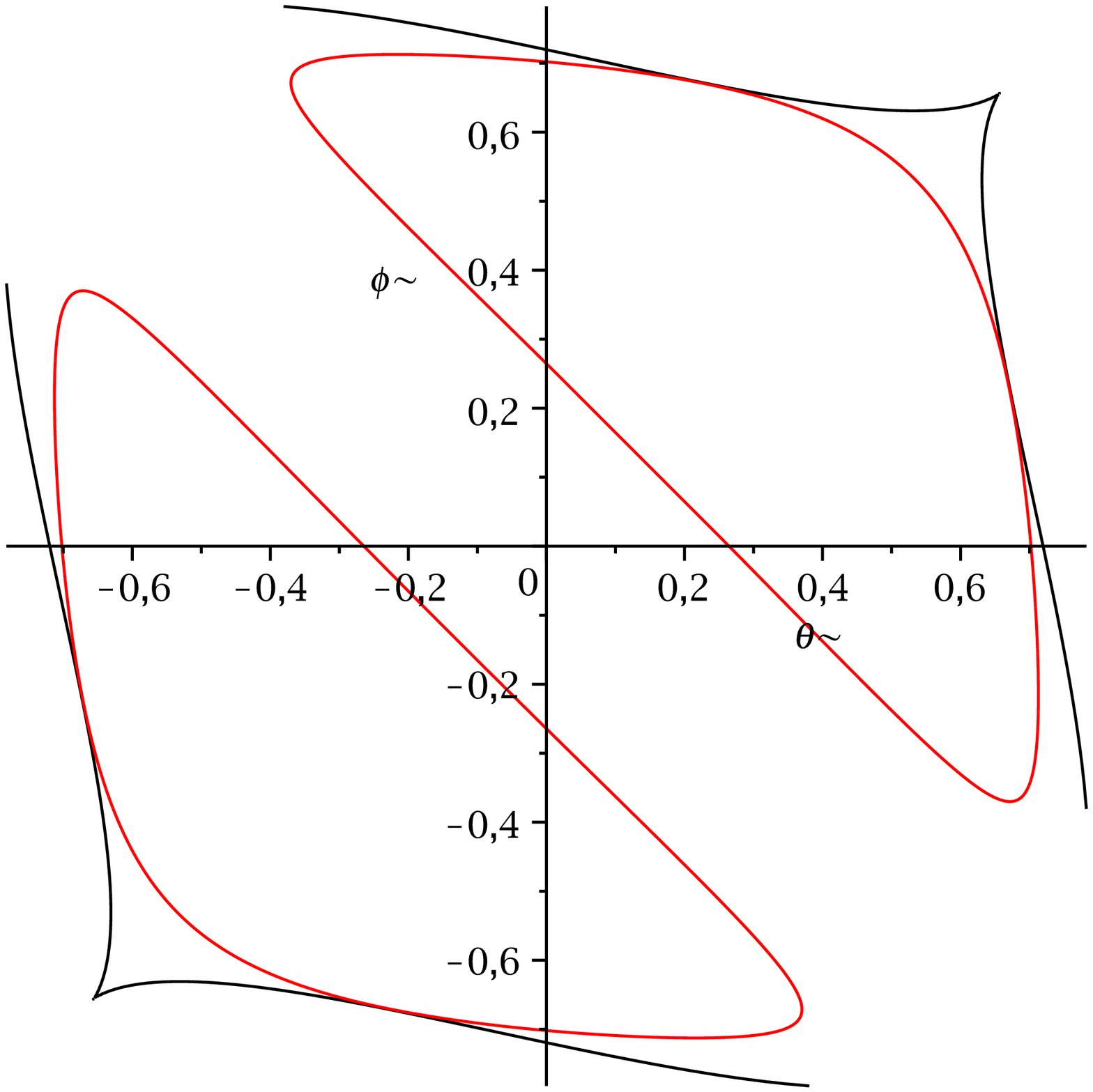}\\
&\\
$\psi=0.04$ & $\psi=0.044$\\
&\\
\includegraphics[scale=0.25]{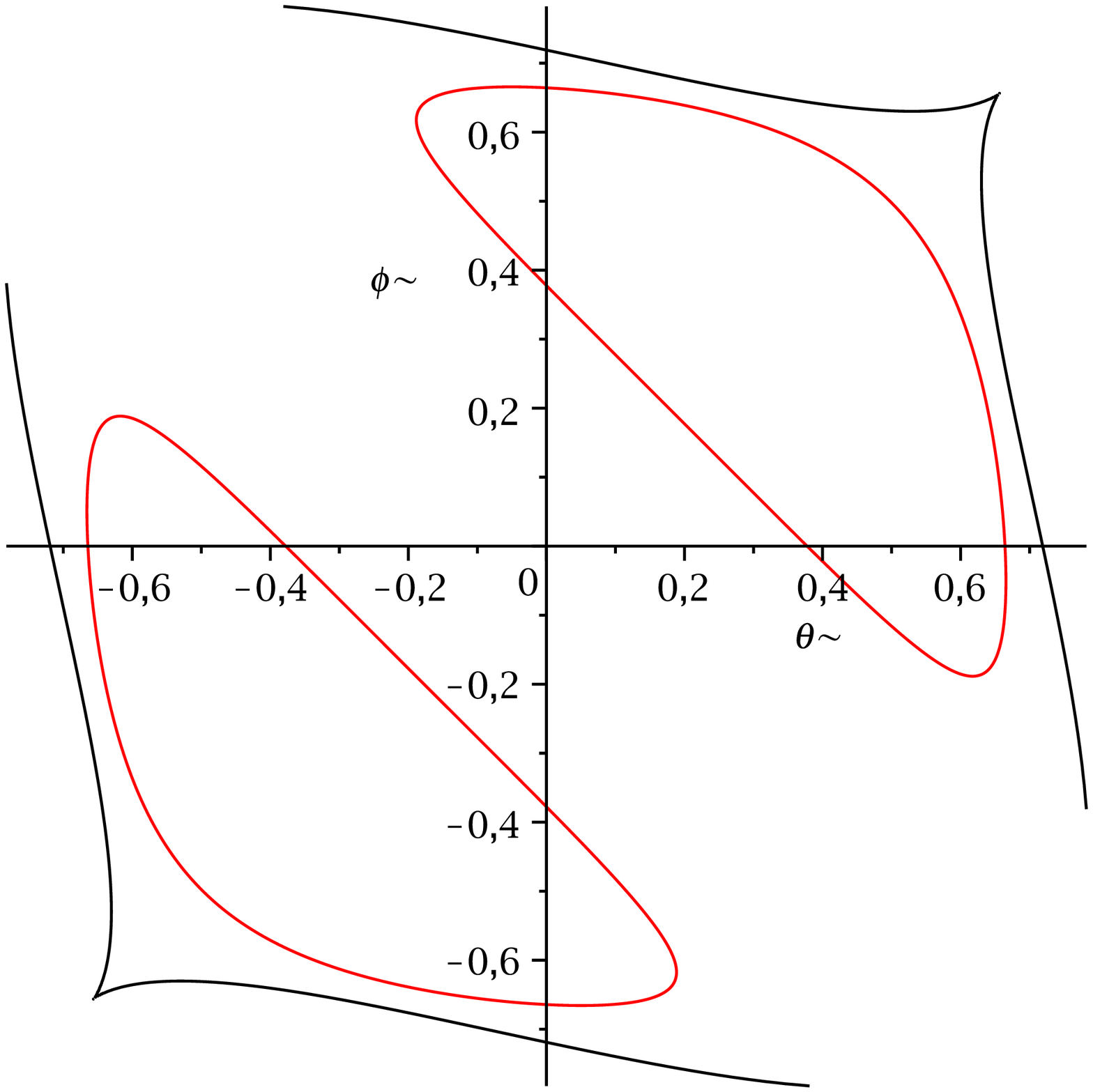}  & 
\includegraphics[scale=0.25]{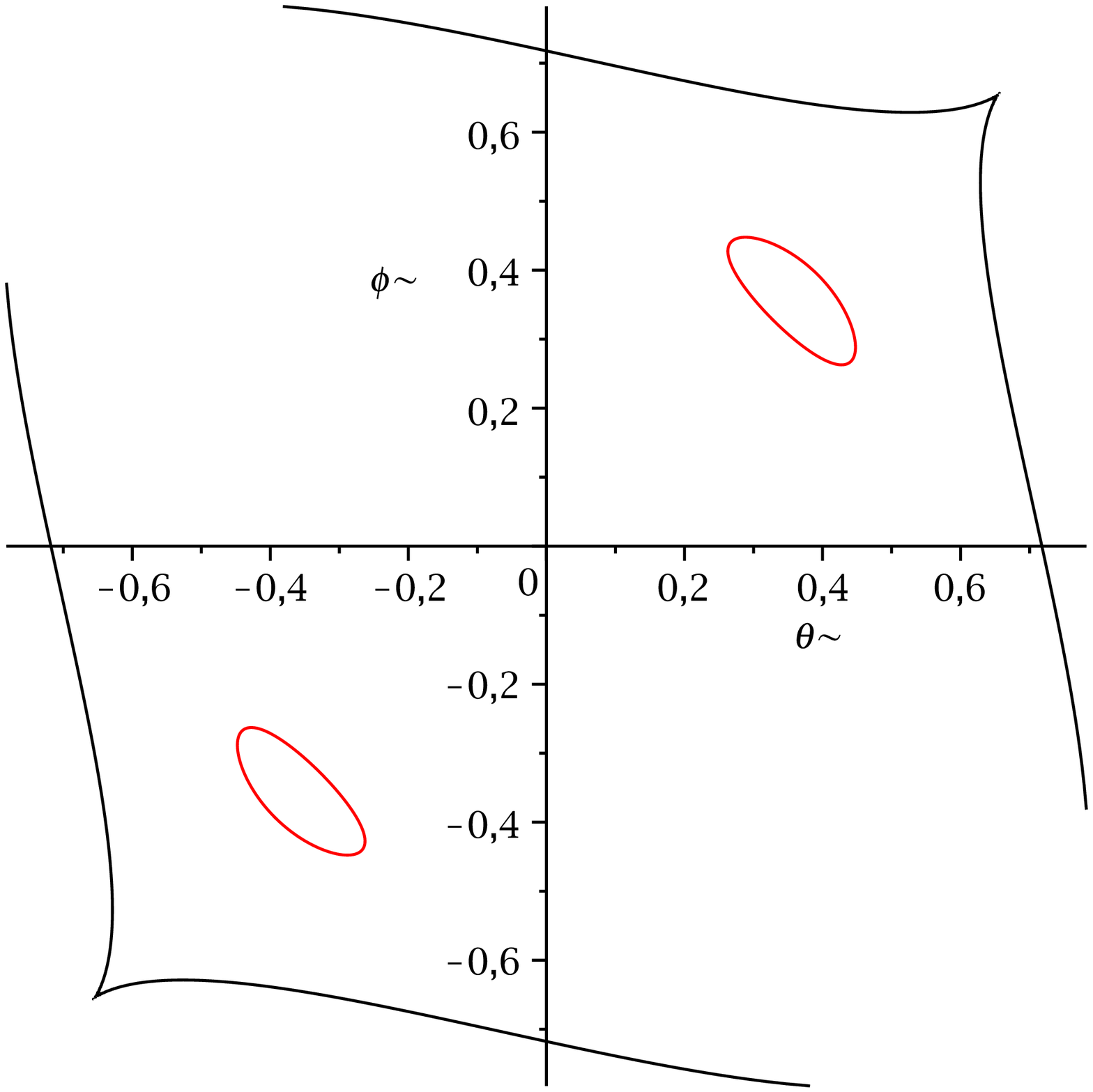}\\
&\\
$\psi=0.06$ & $\psi=0.085$\\
&
\end{tabular}\\
\end{center}
\caption{
The horizontal slice $\psi=\psi_0$ for $\psi_0= 0$, $0.02$, $0.04$, $0.044$,
$0.06$ and $0.085$. The black (resp. red) curve is the intersection of the 
locus where $J_1J_2^{-1}$ (resp. $[J_1,J_2]$) is parabolic. Each intersection 
point corresponds therefore to a group $\la J_1,J_2\ra$ where 
$J_1J_2$, $J_1J_2^{-1}$ and $[J_1,J_2]$ are parabolic. These pictures 
indicate that such groups exists for values of $\psi$ between $0$ and $0.044$.}
\end{figure}

\end{document}